\documentclass[oneside,reqno]{amsart}

\usepackage{hyperref}
\usepackage[totalheight=20 true cm, totalwidth=15 true cm]{geometry}

\usepackage[shortcuts]{extdash}

\usepackage{xypic}

\usepackage{enumerate}
\usepackage{amsthm}
\usepackage{amsmath}
\usepackage{amssymb}
\usepackage{amsfonts}
\usepackage{dsfont}
\usepackage{array}
\usepackage{tabularx}
\usepackage{graphicx}

\newtheorem{theo}{Theorem}[section]
\newtheorem{lemma}[theo]{Lemma}
\newtheorem{prop}[theo]{Proposition}
\newtheorem{cor}[theo]{Corollary}

\newtheorem{remark}[theo]{Remark}

\theoremstyle{definition}
\newtheorem{ex}[theo]{Example}

\newtheorem{rem}[theo]{Remark}

\newcommand{\oo}{\mathcal{O}}
\newcommand{\f}{\phi}

\newcommand{\spec}{\operatorname{Spec}}

\newcommand{\gal}{\operatorname{Gal}}
\newcommand{\Gl}{\operatorname{GL}}

\newcommand{\Aut}{\operatorname{Aut}}
\newcommand{\Hom}{\operatorname{Hom}}

\newcommand{\eX}{\mathfrak{X}}

\newcommand{\id}{\operatorname{id}}

\newcommand{\alg}{\operatorname{alg}}

\newcommand{\nn}{\mathbb{N}}
\newcommand{\C}{\mathbb{C}}
\newcommand{\CC}{\mathcal{C}}
\renewcommand{\AA}{\mathcal{A}}
\newcommand{\BB}{\mathcal{B}}
\newcommand{\DD}{\mathcal{D}}

\newcommand{\rk}{\operatorname{rk}}

\def\Sl{\operatorname{SL}}

\newcommand{\red}{\operatorname{red}}

\newcommand{\Rep}{\operatorname{Rep}}
\newcommand{\uRep}{\bold{Rep}}

\newcommand{\fun}{\operatorname{Fun}}

\newcommand{\loc}{\bold{Loc}}

\newcommand{\Eh}{\mathcal{E}}
\newcommand{\Fh}{\mathcal{F}}
\newcommand{\Th}{\mathcal{T}}
\newcommand{\Fr}{\operatorname{Fr}}
\newcommand{\Sh}{\mathcal{S}}
\newcommand{\ev}{\operatorname{ev}}
\newcommand{\coker}{\operatorname{coker}}
\newcommand{\nVec}{\operatorname{Vec}}
\newcommand{\rel}{\operatorname{rel}}
\newcommand{\ui}{\underline{i}}

\newcommand{\GG}{\mathbb{G}}
\newcommand{\QQ}{\mathbb{Q}}
\newcommand{\ZZ}{\mathbb{Z}}
\newcommand{\Fa}{\mathbb{F}}
\newcommand{\et}{\operatorname{et}}
\newcommand{\oF}{\overline{F}}
\newcommand{\uLoc}{\bold{Loc}}
\newcommand{\Iso}{\operatorname{Iso}}
\newcommand{\silo}{\xrightarrow{\sim}}
\newcommand{\tx}{\tilde{x}}
\newcommand{\colim}{\operatorname{colim}}
\newcommand{\uH}{\underline{H}}
\newcommand{\uAut}{\underline{\operatorname{Aut}}}
\newcommand{\tz}{\tilde{z}}
\newcommand{\rat}{\operatorname{rat}}

\newcommand{\ok}{\overline{k}}
\newcommand{\ox}{\overline{x}}
\newcommand{\uA}{\underline{A}}
\newcommand{\uk}{\underline{k}}
\newcommand{\sep}{\operatorname{sep}}
\newcommand{\Mh}{\mathcal{M}}
\newcommand{\tX}{\tilde{X}}
\newcommand{\ab}{\operatorname{ab}}
\newcommand{\car}{\operatorname{char}}
\newcommand{\Ext}{\operatorname{Ext}}
\newcommand{\oQ}{\overline{\QQ}}
\newcommand{\GL}{\operatorname{GL}}
\newcommand{\beweisende}{\hspace*{\fill}$\Box$}

\makeatletter
\@namedef{subjclassname@2020}{\textup{2020} Mathematics Subject Classification}
\makeatother

\usepackage{todonotes} 

\title[On the proalgebraic fundamental group]{On the proalgebraic fundamental group of topological spaces and amalgamated products of affine group schemes}
\author{Christopher Deninger} \author{Michael Wibmer}

\address{Christopher Deninger, Mathematical Institute, University of M\"unster, Einsteinstr. 62, 48149 M\"unster, Germany, \url{https://www.uni-muenster.de/Arithm/deninger/index.html}}
\email{c.deninger@uni-muenster.de}
\address{Michael Wibmer, Institute of Analysis and Number Therory, Graz University of Technology, Kopernikusgasse~24, 8010 Graz, Austria, \url{https://sites.google.com/view/wibmer}}
\email{wibmer@math.tugraz.at}

\thanks{Christopher Deninger: Funded by the Deutsche Forschungsgemeinschaft (DFG, German Research Foundation) under Germany's Excellence Strategy EXC 2044--390685587, Mathematics M\"unster: Dynamics--Geometry--Structure and the CRC 878 Groups, Geometry \& Actions\\
Michael Wibmer: Supported by the Lise Meitner grant M 2582-N32 of the Austrian Science Fund FWF}


\date{\today}
\dedicatory{\normalsize Dedicated to Massimo Bertolini on the occation of his 60th birthday}
\begin{document}

\begin{abstract}
	The proalgebraic fundamental group of a connected topological space $X$, recently introduced by the first author, is an affine group scheme whose representations classify local systems of finite-dimensional vector spaces on $X$.
	
	In this article, we further develop the theory of the proalgebraic fundamental group, in particular, we establish homotopy invariance and a Seifert-van Kampen theorem. To facilitate the latter, we study amalgamated free product of affine group schemes. We also compute the proalgebraic fundamental group of the arithmetically relevant Kucharcyzk-Scholze spaces and compare it to the motivic Galois group.
\end{abstract}	
	
\maketitle

In \cite{KS}, Kucharczyk and Scholze introduced the \'etale fundamental group $\pi^{\et}_1 (X,x)$ of a pointed connected topological space $(X,x)$, using Grothendieck's formalism of Galois categories. It is a profinite group which classifies the finite coverings of $X$. For path-connected, locally path-connected and semi-locally simply connected spaces $X$, the group $\pi^{\et}_1 (X,x)$ is the profinite completion of the ordinary fundamental group $\pi_1 (X,x)$. For a field $F$ containing the maximal cyclotomic extension $\QQ (\mu_{\infty})$ of $\QQ$, Kucharczyk and Scholze defined conneted topological spaces $\eX_F (\C)$ and $X_F$ whose \'etale fundamental groups can be identified with the absolute Galois group of $F$. These spaces are in general not (locally) path-connected and are rather wild compared to $CW$-complexes for example; hence the need for the introduction of $\pi^{\et}_1 (X,x)$. One could hope to realize not only the Galois group of a field $F$, but even the motivic Galois group of $F$, which for motives over $F$ with coefficients in a field $k$ is a $k$-group i.e. an affine group scheme over $k$. The maximal pro-\'etale quotient of the motivic Galois group is the ordinary Galois group viewed as a $k$-group. For these reasons, in \cite{D} a fundamental $k$-group $\pi_k (X,x)$ for connected pointed topological spaces $(X,x)$ was introduced and studied. It is defined using the Tannakian formalism \cite{DeligneMilne:TannakianCategories}, as the $k$-group whose finite dimensional representations classify the local systems $\Eh$ of finite-dimensional $k$-vector spaces on $X$ with respect to the fibre functor $\Eh \mapsto \Eh_x$. It is shown in \cite{D} that the maximal pro-\'etale quotient of $\pi_k (X,x)$ is $\pi^{\et}_1 (X,x)$ as a $k$-group. For well behaved topological spaces as above, $\pi_k (X,x)$ is the proalgebraic completion of $\pi_1 (X,x)$. Moreover, using ideas of Nori and a theorem of Deligne \cite{Deligne:LetterToAVasiu} on isomorphisms of fibre functors, a replacement for the ordinary universal covering space was constructed in \cite{D} if $k$ is algebraically closed. Only the beginnings of the theory of the proalgebraic fundamental group $\pi_k (X,x)$ were established in \cite{D}. The most glaring omissions were discussions of homotopy invariance and of a Seifert-van Kampen theorem. Moreover the proof in \cite{D} that $\pi_k (X,x)$ is always reduced was not based on a general Tannakian criterion but on a precise study of the special situation. Finally, the proalgebraic fundamental groups of the Kucharczyk-Scholze spaces $\eX_F (\C)$ and $X_F$ were not calculated, although this is a necessary first step to find topological spaces with $\pi_k$ related to motivic Galois groups. In the present paper, we clarify all these items.


For the Seifert-van Kampen theorem, we construct and study amalgamated free products of $k$\=/groups.
It is known \cite{Porst:TheFormalTheoryOfHopfAlgebrasPartII} that the category of commutative Hopf algebras over $k$ is complete. In particular, it has fibre products and so amalgamated free products of $k$-groups exist. However, this existence result offers no insight into the structure and the algebraic properties of amalgamated free products of $k$-groups. Here we offer two constructions of amalgamated free products of $k$\=/groups. The first construction is based on the formalism of Tannakian categories applied to fibre products of categories. It has the advantage that it leads to a proof of the  Seifert-van Kampen theorem for the proalgebraic fundamental group rather directly. The second construction is based on an explicit, albeit quite involved construction with Hopf algebras. It has the advantage that it allows us to deduce properties of the amalgamated free product from properties of the factors. For example, we show that the amalgamated free product of two geometrically reduced $k$-groups is geometrically reduced. 



We conclude the introduction with an overview of the article. In section \ref{sec:1} we construct the amalgamated free product of $k$-groups using the formalism of Tannakian categories applied to fibre products of categories. In section \ref{sec:2} we present an explicit construction for the Hopf algebras of free amalgamated products of $k$-groups and use this to derive several properties of amalgamated free products.
The proof in section \ref{sec:3} of the Seifert-van Kampen theorem for the proalgebraic fundamental group follows immediately from our categorial construction of the amalgamated free product of $k$-groups. In section \ref{sec:4} we prove Tannakian criteria for an affine group scheme over a perfect field $k$ of positive characteristic to be (geometrically) reduced and even stronger to be perfect. Although these criteria are not stated explicitely in \cite{Coulembier:TannakianCategoriesInPositiveCharacteristic}, all the necessary ideas can be found there. As an application we show that the proalgebraic fundamental group $\pi_k (X,x)$ of a connected pointed topological space $(X,x)$ is always perfect. This improves the result of \cite{D} asserting that $ \pi_k (X,x)$ is reduced with an easier and more conceptual proof. In the last section \ref{sec:5} we first prove homotopy invariance of $\pi_k (X,x)$ for all fields $k$. The proof requires some care since we make no assumptions on the topologial space $X$ apart from being connected. We then recall the construction of the spaces $\eX_F (\C)$ and $X_F$ for fields $F \supset \QQ (\mu_{\infty})$ in \cite{KS} and compute $\pi_k$ along the same lines as $\pi^{\et}_1$ in \cite{KS}. Whereas $\pi^{\et}_1$ vanishes if $F$ is algebraically closed this is not true for $\pi_k$. We find that $\pi_k (X_{\oF} , \ox)$ is the identity component $\pi_k (X_F, x)^0$ which turns out to be commutative. The pro-unipotent part of $\pi_k (X_F , x)^0$ is related to the pro-unipotent part of the motivic Galois group for motives over $F$ with coefficients in $k \supset \QQ$. For the pro-reductive quotients there does not seem to be a relation. This shows again that as already noted in the introduction to \cite{KS} the spaces are not yet the optimal ones (in \cite{KS} it was noted that the Steinberg relations do not hold in the rational cohomology of $X_F$). 


\medskip

The authors are grateful to Andy Magid and Hans Porst for helpful comments. The second author is grateful for the invitation to the Cluster of Excellency, Mathematics M\"{u}nster, where this work originated.
%

\hspace{4mm}

\noindent {\bf Conventions.} Throughout this article $k$ is a field. All schemes are assumed to be over
 $k$ unless the contrary is indicated. All rings and algebras are assumed to be unital and commutative. In particular, all Hopf algebras are assumed to be commutative. Unadorned tensor products of $k$\=/vector spaces are understood to be over $k$. For simplicity, we use the term ``$k$-group'' for ``affine group scheme over $k$''. For a $k$-group $G$ we denote its Hopf algebra of global sections by $k[G]$. A morphism $\f\colon G\to H$ of $k$-groups is a quotient map if it is faithfully flat (equivalently, the dual map $\f^*\colon k[H]\to k[G]$ is injective).

\section{Fibre products of Tannakian categories} \label{sec:1}

In this section we discuss fibre products of neutralized Tannakian categories. These 
fibre products can be used to establish the existence of amalgamated free products of $k$-groups (Corollary~\ref{t1-12}) and they play a crucial role in the proof of our Seifert-van Kampen theorem for the proalgebraic fundamental group (Theorem \ref{t32}).

\subsection{Fibre products of categories}

While the product category of two categories is a well-known construction, the fibre product of two categories may be less familiar. We recall the definition from \cite[\href{https://stacks.math.columbia.edu/tag/003O}{Tag 003O}, Example 4.31.3]{stacks-project} or \cite[Chapter VII, \S 3]{Bass:AlgebraicKtheory}.

Let $F\colon \mathcal{A}\to \mathcal{C}$ and $G\colon \mathcal{B}\to \mathcal{C}$ be functors. The \emph{fibre product} $\mathcal{A}\times_\mathcal{C}\mathcal{B}$ is the category defined as follows: Objects are triples  $(A,B,c)$, where $A,B$ are objects of $\mathcal{A}$, $\BB$ and $c\colon F(A)\to G(B)$ is an isomorphism. A morphism $(A,B,c)\to (A',B',c')$ in $\AA\times_\CC \BB$ is a pair $(a,b)$ of morphisms $a\colon A\to A'$, $b\colon B\to B'$ such that
$$
\xymatrix{
	F(A) \ar^-{F(a)}[r] \ar_-{c}[d] & F(A') \ar^-{c'}[d] \\
	G(B) \ar^-{G(b)}[r] & G(B')
}
$$
commutes. The composition of two morphisms $(a,b)\colon (A,B,c)\to (A',B',c')$ and $(a',b')\colon (A',B',c')\to (A'',B'',c'')$ in $\AA\times_\CC \BB$ is $(a'a,b'b)$. If $\CC$ is the trivial category with one object and one morphism (the identity), then there are unique functors $F$ and $G$, and $\AA \times_{\CC} \BB = \AA \times \BB$ is the product of $\AA$ and $\BB$.


\begin{ex}
For an (abstract) group $G$, let $\Rep(G)$ denote the category of finite dimensional $k$-linear representations of $G$.	Let $\f_1\colon H\to G_1$ and $\f_2\colon H\to G_2$ be morphisms of groups. We then have restriction functors $\operatorname{Res}_{\f_1}\colon \Rep(G_1)\to \Rep(H)$ and $\operatorname{Res}_{\f_2}\colon \Rep(G_2)\to \Rep(H)$, which we can use to form the fibre product $\Rep(G_1)\times_{\Rep(H)}\Rep(G_2)$. The free product $G_1*_H G_2$ of $G_1$ and $G_2$ with amalgamation over $H$, comes with a commutative diagram 
$$
\xymatrix{
& G_1*_H G_2 &	\\
G_1 \ar^-{i_1}[ru] & &  G_2 \ar_-{i_2}[lu] \\	
& H \ar^-{\f_1}[lu] \ar_-{\f_2}[ru] &
}
$$
Define a functor $$F\colon \Rep(G_1*_HG_2)\to \Rep(G_1)\times_{\Rep(H)}\Rep(G_2),\ V\rightsquigarrow (\operatorname{Res}_{i_1}(V), \operatorname{Res}_{i_2}(V), \id_V)$$ by $F(f)=(\operatorname{Res_{i_1}(f)}, \operatorname{Res_{i_2}(f)})$ for a morphism $f\colon V\to W$ in $\Rep(G_1*_HG_2)$.
A $k$-linear map $V\to W$ between $G_1*_H G_2$-representation is a $G_1*_H G_2$-morphism if and only if it is a $G_1$\=/morphism and a $G_2$-moprhism. Thus $F$ is fully faithful. To see that $F$ is essentially surjective, consider an object $(V_1,V_2,c)$ of $\Rep(G_1)\times_{\Rep(H)}\Rep(G_2)$. We can use the $k$-linear isomorphism $c\colon V_1\to V_2$ to define a $G_2$-representation on $V_1$. Explicitly, we have $g_2(v_1)=c^{-1}(g_2(c(v_1)))$ for $v_1\in V_1$ and $g_2\in G_2$. Because $c$ is an $H$-morphism, an $h\in H$ acts in this new representation on $V_1$ in the same fashion as it would act through the original $G_1$ representation on $V_1$. That is, the morphisms $G_1\to \Gl(V_1)$ and $G_2\to \Gl(V_1)$ have the same restriction to $H$. The universal property of $G_1*_H G_2$ yields a morphism $G_1*_H G_2\to \Gl(V_1)$, i.e., $V_1$ is a $G_1*_H G_2$-representation with $\operatorname{Res_{i_1}}(V_1)$ the original $G_1$-representation on $V_1$ and $\operatorname{Res_{i_2}}(V_1)\simeq V_2$ as $G_2$-representation via $c$. So $F(V_1)=(V_1,\operatorname{Res_{i_2}}(V_1),\id_{V_1})$ is isomorphic to $(V_1,V_2,c)$ via the isomorphism $(\id_{V_1},c)$ in $\Rep(G_1)\times_{\Rep(H)}\Rep(G_2)$. Thus $F$ is an equivalence of categories. Note that for the trivial group $H$, the category $\Rep H$ is trivial and the free product $G_1 \ast G_2$ corresponds to $\Rep (G_1) \times \Rep (G_2)$. 
\end{ex}

\begin{ex} \label{t1-2}
	For a topological space $X$, let $\operatorname{Sh}(X)$ denote the category of sheaves of abelian groups on $X$. Let $U_1, U_2$ be open subset of $X$ such that $X=U_1\cup U_2$. We have restriction functors $\operatorname{Sh}(U_1)\to \operatorname{Sh(U_1\cap U_2)},\ \mathcal{F}\rightsquigarrow \mathcal{F}|_{U_1\cap U_2}$ and $\operatorname{Sh}(U_2)\to \operatorname{Sh(U_1\cap U_2)},\ \mathcal{F}\rightsquigarrow\mathcal{F}|_{U_1\cap U_2}$, which we can use to form the fibre product $\operatorname{Sh}(U_1)\times_{\operatorname{Sh}(U_1\cap U_2)} \operatorname{Sh}(U_2)$.
	Basic sheaf theory (\cite[\href{https://stacks.math.columbia.edu/tag/00AK}{Tag 00AK}]{stacks-project}) implies that $$\operatorname{Sh(X)}\to \operatorname{Sh}(U_1)\times_{\operatorname{Sh}(U_1\cap U_2)} \operatorname{Sh}(U_2),\ \mathcal{F}\rightsquigarrow (\mathcal{F}|_{U_1},\mathcal{F}|_{U_2}, \id_{\mathcal{F}|_{U_1\cap U_2}})$$
	is an equivalence of categories. Note that since $\Fh (\emptyset) = 0$ for any sheaf of abelian groups, $\operatorname{Sh}(\emptyset)$ is the trivial category. So in case $U_1\cap U_2$ is empty, the above construction yields the product category.
\end{ex}


%

We return to the discussion of the general case. We have a functor $P_1\colon \mathcal{A}\times_\mathcal{C}\mathcal{B}\to \AA$ given by $P_1(A,B,c)=A$ and $P_1(a,b)=a$. Similarly, $P_2\colon \mathcal{A}\times_\mathcal{C}\mathcal{B}\to \BB$ is given by $P_2(A,B,c)=B$ and $P_2(a,b)=b$. The diagram
$$ 
\xymatrix{
	& \AA\times_\CC\BB \ar_-{P_1}[ld]  \ar^-{P_2}[rd]  & \\
	\AA \ar_-F[rd] & & \BB \ar^-G[ld] \\
	& \CC &
}
$$
$2$-commutes in the sense that the functors $FP_1$ and $GP_2$ are isomorphic via an isomorphism $\psi\colon FP_1\simeq G P_2$. In fact, for any object $(A,B,c)$ of $\AA\times_\CC\BB$ we have an isomorphism 
$$ \psi_{(A,B,c)}\colon(FP_1)(A,B,c)=F(A)\xrightarrow{c} G(B)=(GP_2)(A,B,c).$$

To state the ($2$-categorical) universal property satisfied by the fibre product $\AA\times_\CC\BB$, we formaly introduce $2$-commutative diagrams as follows. A \emph{$2$-commutative diagram} (with respect to $F\colon \AA\to \CC$ and $G\colon \BB\to \CC$) is a quadruple $(\DD, Q_1,Q_2,\alpha)$, where $Q_1\colon \DD\to \AA$ and $Q_1\colon \DD\to\BB$ are functors and $\alpha\colon FQ_1\simeq GQ_2$ is an isomorphism of functors. A morphism 
$(\DD, Q_1,Q_2,\alpha)\to (\DD', Q'_1,Q'_2,\alpha')$ of $2$-commutative diagrams is a triple $(H,\beta_1,\beta_2)$, with $H\colon \DD\to \DD'$ a functor and $\beta_1\colon Q_1\to Q_1'H$ and $\beta_2\colon Q_2\to Q_2'H$ isomorphisms of functors, such that 
$$
\xymatrix{
FQ_1 \ar[r]^-{\id_F\star\beta_1} \ar_-{\alpha}[d] & FQ_1'H \ar^-{\alpha'\star\id_H}[d] \\
GQ_2 \ar[r]^-{\id_G\star\beta_2} & GQ_2'H	
}
$$
commutes. An isomorphism between two morphisms $(H,\beta_1,\beta_2)$, $(H',\beta_1', \beta_2')$ (with the same source and target) is an isomorphism $\theta\colon H\to H'$ such that
$$
\xymatrix{
Q_1 \ar^-{\beta_1}[r] \ar_{\beta_1'}[rd] & Q_1'H \ar^-{\id_{Q_1'}\star\theta}[d] \\
& Q_1'H'	
}
\quad \text{ and } \quad
\xymatrix{
	Q_2 \ar^-{\beta_2}[r] \ar_{\beta_2'}[rd] & Q_2'H \ar^-{\id_{Q_2'}\star\theta}[d] \\
	& Q_2'H'	
}
$$
commute.
The fibre product of categories satisfies the following ($2$-categorical) universal property.

\begin{lemma} \label{lemma: universal property of fibre product}
	Let $F\colon \mathcal{A}\to \mathcal{C}$ and $G\colon \mathcal{B}\to \mathcal{C}$ be functors.
	For every $2$-commutative diagram $(\DD,Q_1,Q_2,\alpha)$ there exists a morphism  $(H,\beta_1,\beta_2)\colon (\DD,Q_1,Q_2,\alpha)\to (\AA\times_\CC\BB, P_1,P_2,\psi)$ of $2$\=/commutative diagrams. If $(H',\beta_1',\beta_2')$ is another such morphism, then there exists a unique isomorphism $\theta$ between $(H,\beta_1,\beta_2)$ and $(H',\beta_1',\beta_2')$.
\end{lemma}
\begin{proof}
	This is \cite[\href{https://stacks.math.columbia.edu/tag/02X9}{Tag 02X9}]{stacks-project}. 
	For what follows, it will be helpful to have explicit formulas for $(H,\beta_1,\beta_2)$ and $\theta$. On objects let $H$ be given by $H(D)=(Q_1(D),Q_2(D),\alpha_D)$ and for a morphism $f\colon D\to D'$ in $\DD$ set $H(f)=(Q_1(f),Q_2(f))\colon (Q_1(D),Q_2(D),\alpha_D)\to (Q_1(D'),Q_2(D'),\alpha_{D'})$. As $P_1H=Q_1$, $\beta_1\colon Q_1\simeq  P_1H$ can be chosen to be the identity. Similarly, $\beta_2\colon Q_2\simeq P_2H$ is the identity.
	
	For another morphism $(H',\beta'_1,\beta'_2)\colon (\DD,Q_1,Q_2,\alpha)\to (\AA\times_\CC\BB, P_1,P_2,\psi)$ of $2$\=/commutative diagrams, the commutativity of 
	$$
	\xymatrix{
		FQ_1 \ar[r]^-{\id_F\star\beta'_1} \ar_-{\alpha}[d] & FP_1H' \ar^-{\psi\star\id_{H'}}[d] \\
		GQ_2 \ar[r]^-{\id_G\star\beta'_2} & GP_2H'	
	}
	$$
	shows that for any object $D$ of $\DD$ we have a morphism $$\theta_D=(\beta'_{1,D},\beta'_{2,D})\colon H(D)=(Q_1(D),Q_2(D),\alpha_D)\longrightarrow H'(D)=(P_1(H'(D)),P_2(H'(D)),\alpha'_D)$$
	in $\AA\times_{\CC}\BB$.
\end{proof}




Before considering fibre products of Tannakian categories, we need to a take a brief look at fibre products of abelian categories.

\begin{lemma} \label{lemma: fibre product of abelian}
	If $F\colon\AA\to\CC$ and $G\colon \BB\to\CC$ are exact additive functors between abelian categories, then $\AA\times_\CC\BB$ is an abelian category.
\end{lemma}
\begin{proof}
	Let us begin by defining the addition on the hom-sets of $\AA\times_\CC \BB$. Let $(a,b),(a',b')\colon (A,B,c)\to (A',B',c')$ be two morphisms in $\AA\times_\CC\BB$. Adding the identities $G(b)c=c'F(a)$ and $G(b')c=c'F(a')$ and using the additivity of $F$ and $G$ yields $G(b+b')c=c'F(a+a')$. So we can define $(a,b)+(a',b')=(a+a',b+b')$. This defines the structure of an abelian group on the hom-sets such that composition is bilinear, i.e., $\AA\times_\CC \BB$ is a pre-additive category.
	
	 A zero object of $\AA\times_\CC \BB$ is given by $(0_\AA,0_\BB,0)$, where $0_\AA$ and $0_\BB$ are zero objects of $\AA$ and $\BB$ respectively and $0\colon F(0_\AA)\to G(0_\BB)$ is the zero morphism. The direct sum of $(A,B,c)$ and $(A',B',c')$ is $(A\oplus A',B\oplus B',c\oplus c')$. So $\AA\times_\CC \BB$ is an additive category.
	 
	 Let $(a,b)\colon (A,B,c)\to (A',B',c')$ be a morphism in $\AA\times_\CC\BB$. Since $F$ and $G$ preserve kernels, there exists a unique isomorphism $c_{\ker}\colon F(\ker(a))\to G(\ker(b))$ such that
	 $$
	 \xymatrix{
	 F(\ker(a))\ar[r] \ar_-{c_{\ker}}@{..>}[d] & F(A) \ar^-{F(a)}[r] \ar_-{c}[d] & F(A')  \ar^-{c'}[d] \\
	 G(\ker(b)) \ar[r] & G(B) \ar^-{G(b)}[r] & G(B')	
	 }
	 $$
	 commutes. The resulting morphism $(\ker(a),\ker(b),c_{\ker})\to (A,B,c)$ in $\AA\times_\CC\BB$ is then a kernel of $(a,b)$. A similar construction shows that cokernels exist in $\AA\times_\CC\BB$. Since kernels and cokernels are constructed componentwise, it follows that for any morphism $(a,b)$ in $\AA\times_\CC\BB$ the morphism $\operatorname{Coim}(a,b)\to \operatorname{Im}(a,b)$ is an isomorphism. Thus $\AA\times_\CC \BB$ is abelian.	 
\end{proof}

\subsection{Tannakian categories}
Our next goal is to show that the fibre product of Tannakian categories is naturally a Tannakian category. To this end, we first recall the basic facts concerning Tannakiann categories. For more background see \cite{DeligneMilne:TannakianCategories} and \cite{Deligne:categoriestannakien}.

A \emph{tensor category} is a tuple $(\AA,\otimes,\f,\psi)$, where $\AA$ is a category, $\otimes\colon \AA\times\AA\to \AA,\ (X,Y)\rightsquigarrow X\otimes Y$ is a functor and $\f$, $\psi$ are isomophisms of functors expressing the associativity and commutativity of $\otimes$. More precisely, the \emph{associativity constraint} $\f$ has components $\f_{X,Y,Z}\colon X\otimes (Y\otimes Z)\to(X\otimes Y)\otimes Z$ and the \emph{commutativity constraint} $\psi$ has components $\psi_{X,Y}\colon X\otimes Y\to Y\otimes X$. The functorial isomorphisms $\f$ and $\psi$ are required to satisfy three \emph{coherence conditions} (see \cite[Section~1]{DeligneMilne:TannakianCategories} for details). It is also required that there exists an \emph{identity object} $(\mathds{1},u)$. This means that $\AA\to\AA,\ X\rightsquigarrow \mathds{1}\otimes X$ is an equivalence of categories and $u\colon \mathds{1}\to\mathds{1}\otimes\mathds{1}$ is an isomorphism. 

Then there exists a unique isomorphism of functors $l$ with component $l_X\colon X\to \mathds{1}\otimes X$ such that certain properties are satisfied (\cite[Prop. 1.3 (a)]{DeligneMilne:TannakianCategories}). Similarly, there is functorial isomorphism $r$ with components $r_X\colon X\to X\otimes\mathds{1}$.

We will usually omit $\f$ and $\psi$ from the notation and refer to $(\AA,\otimes)$, or even $\AA$, as a tensor category. A \emph{tensor functor} $(\AA,\otimes)\to (\AA',\otimes')$ is a pair $(F,\alpha)$, where $F\colon \AA\to \AA'$ is a functor and $\alpha$ is an isomorphism of functors with components $\alpha_{X,Y}\colon F(X)\otimes F(Y)\to F(X\otimes Y)$ such that three natural properties involving respectively the associativity constraints, the commutativity constraints and the identity objects, are satisfied (\cite[Def. 1.8]{DeligneMilne:TannakianCategories}). We will often omit the functorial isomorphism $\alpha$ from the notation and simply refer to $F$ as a tensor functor.

A tensor category $(\AA,\otimes)$ is \emph{rigid}, if for every object $X$ of $\AA$ the exists an object $X^\vee$ (a \emph{dual} of $X$) of $\AA$ together with morphisms $\operatorname{ev}\colon X\otimes X^\vee\to \mathds{1}$ and $\delta\colon\mathds{1}\to X^\vee\otimes X$ such that

$$X\xrightarrow{r} X\otimes\mathds{1}\xrightarrow{\id\otimes \delta}X\otimes (X^\vee\otimes X)\xrightarrow{\f}(X\otimes X^\vee)\otimes X\xrightarrow{\operatorname{ev}\otimes\id}\mathds{1}\otimes X\xrightarrow{l^{-1}}X$$
and 
$$X^\vee\xrightarrow{l}\mathds{1}\otimes X^\vee\xrightarrow{\delta\otimes\id}(X^\vee\otimes X)\otimes X^\vee\xrightarrow{\f^{-1}}X^\vee\otimes(X\otimes X^\vee)\xrightarrow{\id\otimes\operatorname{ev}}X^\vee\otimes\mathds{1}\xrightarrow{r^{-1}}X^\vee$$
are the identity. The above definition of rigid is equivalent to the one used in \cite{DeligneMilne:TannakianCategories} by \cite[Section 2]{Deligne:categoriestannakien}.
Given $X, X^\vee$ and $\operatorname{ev}$ such that there exists a $\delta$ verifying the above conditions, then $\delta$ is uniquely determined. Moreover, the pair $(X^\vee,\operatorname{ev})$ is unique up to a unique isomorphism. For a morphism $f\colon X\to Y$ the \emph{transpose} $^{t}f\colon Y^\vee\to X^\vee$ of $f$ is defined as the composition
$${^t}f\colon Y^\vee\simeq\mathds{1}\otimes Y^\vee\xrightarrow{\delta\otimes \id}X^\vee\otimes X\otimes Y^\vee\xrightarrow{\id\otimes f\otimes \id}X^\vee\otimes Y\otimes Y^\vee\xrightarrow{\id\otimes\operatorname{ev}} X^\vee\otimes\mathds{1}\simeq X^\vee.$$
For $f$ invertible one sets $f^\vee=({^t f})^{-1}\colon X^\vee\to Y^\vee$.

\begin{lemma} \label{lemma: tensor on fibre product of tensor categories}
	Let $F\colon \AA\to\CC$ and $G\colon \BB\to\CC$ be tensor functors between tensor categories. Then $\AA\times_\CC\BB$ naturally acquires the structure of a tensor category such that the projections  $\AA\times_\CC\BB\to \AA$ and  $\AA\times_\CC\BB\to \BB$ are tensor functors. Moreover, if $\AA$, $\BB$ and $\CC$ are rigid, then $\AA\otimes_\CC\BB$ is rigid.
\end{lemma}
\begin{proof}
Cf. the discussion on page 360 of \cite{Bass:AlgebraicKtheory}.
	We define the tensor product $\otimes \colon (\AA\times_\CC\BB) \times (\AA\times_\CC\BB)\to \AA\times_\CC\BB$ on objects by $$(A,B,c)\otimes (A',B',c')=(A\otimes A', B\otimes B', c\otimes c')$$ and on morphisms by $(a,b)\otimes (a',b')=(a\otimes a',b\otimes b')$. The associativity and commutativity constraints on $\AA$ and $\BB$ induce associativity and commutativity constraints on $\AA\times_\CC\BB$. The identity object of $\AA\times_\CC\BB$ is $(\mathds{1}_\AA,\mathds{1}_\BB, c)$, with $c\colon F(\mathds{1}_\AA)\to G(\mathds{1}_\BB)$ the unique isomorphism derived from the fact that $F(\mathds{1}_\AA)$ and $G(\mathds{1}_\BB)$ are identity objects of $\CC$.
	
	If $(A,B,c)$ is an object of $\AA\times_\CC\BB$ and $A^\vee$ and $B^\vee$ are duals of $A$ and $B$ respectively, then
	$(A^\vee,B^\vee, c^\vee)$ is a dual of $(A,B,c)$. Implicit in this notation are the identifications $F(A^\vee)\simeq F(A)^\vee$ and $G(B^\vee)\simeq G(B)^\vee$ derived from the fact that tensor functors preserve duals.
\end{proof}

A rigid tensor category $(\AA,\otimes)$ is \emph{abelian} if the underlying category $\AA$ is abelian. In this case $\otimes$ is automatically bi-additive (\cite[Prop. 1.16]{DeligneMilne:TannakianCategories}). 

A basic example of a tensor category is the category $\operatorname{Mod}_R$ of $R$-modules over a commutative ring $R$, with the usual tensor product and the usual associativity and commutativity constraints.
The most basic example of a rigid abelian tensor category is the category $\operatorname{Vec}_k$ of finite dimensional $k$-vector spaces. 

A \emph{neutral Tannakian category} over $k$ is a rigid abelian tensor category $(\AA,\otimes)$ together with an isomorphism $k\simeq \operatorname{End}(\mathds{1})$ such that there exists an exact $k$-linear tensor functor $\omega\colon\AA\to\operatorname{Vec}_k$. Any such $\omega$ is called a \emph{(neutral) fibre functor}. We note that by \cite[Cor. 2.10]{Deligne:categoriestannakien} a fibre functor is automatically faithful.
%

For tensor functors $F,G\colon (\AA,\otimes)\to (\BB,\otimes)$, a morphism $\alpha\colon F\to G$ of functors, is a \emph{morphism of tensor functors} if 
\begin{enumerate}
	\item $\alpha_{X\otimes Y}=\alpha_X\otimes\alpha_Y$ for any two objects $X, Y$ of $\AA$ under the identifications $F(X)\otimes F(Y)\simeq F(X\otimes Y)$ and $G(X)\otimes G(Y)\simeq G(X\otimes Y)$ and
	\item the diagram
	$$
	\xymatrix{
	& \mathds{1}_\BB \ar[ld] \ar[rd] & 	\\
	F(\mathds{1}_\AA) \ar^-{\alpha_{\mathds{1}_\AA}}[rr] & & G(\mathds{1_\AA})
	}
	$$
	commutes.	
\end{enumerate}

For a neutral Tannakian category $\AA$ with fibre functor $\omega\colon\AA\to\operatorname{Vec}_k$, one defines a functor  $\underline{\Aut}^\otimes(\omega)$ from the category of $k$-algebras to the category of groups by $\underline{\Aut}^\otimes(\omega)(R)=\Aut^\otimes(\omega_R)$, the group of tensor automorphisms of the tensor functor $\omega_R$, defined as the composition of $\omega$ with the tensor functor $\operatorname{Vec}_k\to \operatorname{Mod}_R,\ V\rightsquigarrow V\otimes_k R$ for any $k$-algebra $R$.

The standard example of a neutral Tannakian category is the category $\Rep(G)$ of finite dimensional $k$-linear representations of a $k$-group $G$ with the usual tensor product of representations and the forgetful functor $\omega_G\colon \Rep(G)\to \operatorname{Vec}_k$ as fibre functor. The main Tannaka reconstruction theorem shows that this is essentially the only example:

\begin{theo}[{\cite[Theorem 2.11]{DeligneMilne:TannakianCategories}}] \label{theo:DM}
	Let $(\AA,\otimes)$ be a neutral Tannakian category with fibre functor $\omega\colon \AA\to \operatorname{Vec}_k$. Then $G=\underline{\Aut}^\otimes(\omega)$ is a $k$-group and $\omega$ induces an equivalence of tensor categories $\AA\to \Rep(G)$.
\end{theo}

The $k$-group $G_\AA=\underline{\Aut}^\otimes(\omega)$ is called the \emph{fundamental group} of $\AA$ or the \emph{Tannakian dual} of $\AA$.
A \emph{neutralized Tannakian category} is a pair $(\AA,\omega_\AA)$ consisting of a neutral Tannakian category $\AA$ and a fibre functor $\omega_\AA\colon \AA\to \operatorname{Vec}_k$. Abusing notation, a \emph{tensor functor} $(F,\gamma)\colon (\AA,\omega_\AA)\to (\BB,\omega_\BB)$ between neutralized Tannakian categories is a pair $(F,\gamma)$ consisting of a $k$-linear tensor functor $F\colon \AA\to \BB$ and an isomorphism $\gamma\colon \omega_\BB F\simeq \omega_\AA$ of tensor functors. Note that this implies that $F$ is exact since $\omega_\AA$ and $\omega_\BB$ are exact.
Such a tensor functor induces a morphism $\varphi\colon \underline{\Aut}^\otimes(\omega_\BB)\to \underline{\Aut}^\otimes(\omega_\AA)$
of $k$-groups as follows: For every $k$-algebra $R$, an element $g\in\underline{\Aut}^\otimes(\omega_\BB)(R)=\Aut^\otimes(\omega_{\BB,R})$ gives rises to a tensor automorphism $g\star\id_F\colon  \omega_{\BB,R}F\simeq \omega_{\BB,R}F$, that yields a tensor automorphism of $\omega_{\AA,R}$ via the tensor isomorphism $\gamma_R=\id_{-\otimes R}\star\gamma\colon \omega_{\BB,R}F\simeq\omega_{\AA,R}$, i.e., $\varphi_R(g)=\gamma_R (g\star \id_F)\gamma_R^{-1}$.
 Note that in the context of Theorem \ref{theo:DM} we have an isomorphism $(\AA,\omega)\to (\Rep(G),\omega_G)$ of neutralized Tannakian categories.

The composition $(G,\delta)(F,\gamma)$ of two tensor functors  $(F,\gamma)\colon (\AA,\omega_\AA)\to (\BB,\omega_\BB)$ and $(G,\delta)\colon (\BB,\omega_\BB)\to (\CC,\omega_\CC)$ is defined as  $(G,\delta)(F,\gamma)=(GF,\omega_\CC G F\xrightarrow{\delta\star\id_{F}}\omega_\BB F\xrightarrow{\gamma}\omega_\AA)$.

Let $(F,\gamma), (F',\gamma')\colon   (\AA,\omega_\AA)\to (\BB,\omega_\BB)$ be tensor functors. An isomorphism between $(F,\gamma)$ and $(F',\gamma')$ is an isomorphism $\delta\colon F\to F'$ of tensor functors such that 
$$
\xymatrix{
\omega_\BB F \ar^-{\id_{\omega_\BB}\star\delta}[rr] \ar_-{\gamma}[rd] & & \omega_\BB F' \ar^-{\gamma'}[ld] \\
& \omega_\AA &	
}
$$
commutes. Note that isomorphic $(F,\gamma)$, $(F',\gamma')$ induce the same morphism $\underline{\Aut}^\otimes(\omega_\BB)\to \underline{\Aut}^\otimes(\omega_\AA)$.

Morally, the category of $k$-groups is anti-equivalent to the category of neutralized Tannakian categories over $k$. In reality, this is not quite true since a category of categories should be regarded as a $2$-category. The following precise statement (\cite[Cor. 3.2]{OvchinnikovWibmer:TannakianCategoriesWithSemigroupActions}) is sufficient for our purposes.

\begin{lemma} \label{lemma: morphisms of groups and isom classes of functors}
	Let $(\AA,\omega_\AA)$ and $(\BB,\omega_\BB)$ be neutralized Tannakian categories with fundamental groups $G_\AA$ and $G_\BB$ respectively. Then $\Hom(G_\BB,G_\AA)$ is in bijection with the isomorphism classes of tensor functors $(\AA,\omega_\AA)\to(\BB,\omega_\BB)$.
\end{lemma}

As one may expect, the fibre product of neutralized Tannakian categories inherits all of the structure:

\begin{lemma} \label{lemma: tensor on fibre product neutralized Tannakian categories}
	Let $(F,\gamma)\colon (\AA,\omega_{\AA})\to (\CC,\omega_{\CC})$ and $(G,\delta)\colon (\BB,\omega_{\BB})\to(\CC,\omega_{\CC})$ be tensor functors between neutralized Tannakian categories. Then $\AA\times_\CC\BB$ inherits the structure of a neutralized Tannakian category with fibre functor $\omega_{\AA\times_\CC\BB}$ given by  $\omega_{\AA\times_\CC\BB}(A,B,c)=\omega_\AA(A)$.
\end{lemma}
\begin{proof}
We already saw in Lemma \ref{lemma: tensor on fibre product of tensor categories} that $\AA\times_\CC \BB$ is naturally a rigid tensor category.
As $F$ and $G$ are exact and additive, Lemma \ref{lemma: fibre product of abelian} shows that $\AA\times_\CC\BB$ is abelian.

The map $k\to\operatorname{End}(\mathds{1}_{\AA\times_\CC\BB})$ is an isomorphism because an endomorphism of $\mathds{1}_{\AA\times_\CC\BB}=(\mathds{1}_\AA,\mathds{1}_\BB,c)$ is a pair $(a,b)$ with $a\in\operatorname{End}(\mathds{1}_\AA)\simeq k$ and $b\in \operatorname{End}(\mathds{1}_\BB)\simeq k$ such that $a$ and $b$ agree as elements of~$k$.
%

Note that the projection $P_1\colon \AA\times_\CC\BB\to \AA$ is an exact $k$-linear tensor functor (the implied isomorphism $P_1(-)\otimes P_1( -)\simeq P_1(-\otimes -)$ is the identity). Therefore, also the composition 
 $\omega_{\AA\times_\CC\BB}$ of $\omega$ with $P_1$ is an exact $k$-linear tensor functor. 
\end{proof}

\begin{rem}
	There are four possibilities to define the fibre functor $\omega_{\AA\times_\CC\BB}$ in Lemma \ref{lemma: tensor on fibre product neutralized Tannakian categories}: 
	On an object $(A,B,c)$ of $\AA\times_\CC\BB$ we can define it as $\omega_\AA(A)$, $\omega_\BB(B)$, $\omega_\CC(F(A))$ or $\omega_\CC(G(B))$. The choice is irrelevant, since all four functors are isomorphic.
\end{rem}

\subsection{Amalgamated free products}

In this section we recall the definition of amalgamated free products and discuss their relation with fibre products of neutralized Tannakian categories.

Let $\f_1\colon H\to G_1$ and $\f_2\colon H\to G_2$ be morphisms of $k$-groups. In line with the common notation in group theory, we denote a pushout of $\f_1$ and $\f_2$ in the category of $k$-groups with $G_1*_H G_2$ and call it a \emph{free product of $G_1$ and $G_2$ with amalgamation over $H$} or simply an \emph{amalgamated free product}. So $G_1*_H G_2$ comes with two morphisms $i_1\colon G_1\to G_1*_H G_2$ and $i_2\colon G_2\to G_1*_H G_2$ such that  
$$
\xymatrix{
& G_1*_H G_2 & \\
G_1 \ar^-{i_1}[ru] & & G_2 \ar_-{i_2}[lu] \\
& H \ar^{\f_1}[lu] \ar_{\f_2}[ru] &
}
$$
commutes and satisfies the following universal property: If $j_1\colon G_1\to G$ and $j_2\colon G_2\to G$ are morphisms of $k$-groups such that $j_1\f_1=j_2\f_2$, then there exists a unique morphism $G_1*_H G_2\to G$ such that 

$$
\xymatrix{
	G_1 \ar^-{i_1}[r] \ar_-{j_1}[rd] & G_1*_H G_2 \ar@{..>}[d] & G_2 \ar_-{i_2}[l] \ar^{j_2}[ld] \\
	& G &	
}
$$
commutes. 
For $H=1$ the trivial group, one obtains the \emph{free product} $G_1*G_2$ of $G_1$ and $G_2$.

  It is known  (\cite[Prop. 22 (2)]{Porst:TheFormalTheoryOfHopfAlgebrasPartII}) that the category of (commutative) Hopf algebras over a field is complete. In particular, it has fibre products. Therefore amalgamated free products exist. Beyond this existence result, the only literature on amalgamated free products of $k$-groups we could find is \cite{Antei:PushoutOfQuasiFiniteAndFlatGroupschemesOverADedekindRing}, which discusses amalgamated free products of finite flat group schemes over complete discrete valuation rings. We will provide two existence proofs for amalgamated free products of $k$-groups independent from \cite{Porst:TheFormalTheoryOfHopfAlgebrasPartII}. The first one (Corollary \ref{t1-12}), based on the fibre product of neutralized Tannakian categories, leads to a proof of the Seifert-van Kampen Theorem for the proalgebraic fundamental group rather directly. The second one (Proposition~\ref{prop: fibre product of Hopf algebras}), based on an explicit construction with Hopf algebras, is helpful for determining properties of the amalgamated free product.


We next discuss the universal property of the fibre product constructed in Lemma \ref{lemma: tensor on fibre product neutralized Tannakian categories}.
As before, let $(F,\gamma)\colon (\AA,\omega_{\AA})\to (\CC,\omega_{\CC})$ and $(G,\delta)\colon (\BB,\omega_{\BB})\to(\CC,\omega_{\CC})$ be tensor functors between neutralized Tannakian categories. We consider $\AA\times_\CC\BB$ as a neutralized Tannakian category as described in Lemma \ref{lemma: tensor on fibre product neutralized Tannakian categories}
The projections $P_1\colon \AA\times_\CC\BB\to \AA$ and  $P_2\colon \AA\times_\CC\BB\to \BB$ are $k$-linear tensor functors.

%

In fact, $P_1$ and $P_2$ can be regarded as tensor functors between neutralized Tannakian categories as follows. With $\chi_1=\id\colon \omega_\AA P_1\to\omega_{\AA\times_\CC\BB}$ we obtain a tensor functor $(P_1,\chi_1)\colon (\AA\times_\CC\BB,\omega_{\AA\times_\CC\BB})\to (\AA,\omega_\AA)$ between neutralized Tannakian categories. We define an isomorphism $\chi_2\colon \omega_\BB P_2\to \omega_{\AA\times_\CC\BB}$ by
$$\chi_{2,(A,B,c)}\colon (\omega_{\BB}P_2)(A,B,c)=\omega_\BB(B)\xrightarrow{\delta_B^{-1}}\omega_\CC(G(B))\xrightarrow{\omega_\CC(c)^{-1}}\omega_\CC(F(A))\xrightarrow{\gamma_A}\omega_{\AA}(A)=\omega_{\AA\times_\CC\BB}(A,B,c)$$
for any object $(A,B,c)$ of $\AA\times_\CC\BB$. Then also $(P_2,\chi_2)\colon (\AA\times_\CC\BB,\omega_{\AA\times_\CC\BB})\to (\BB,\omega_\BB)$ is a tensor functor between neutralized Tannakian categories. The isomorphism $\psi\colon FP_1\to GP_2$ is an isomorphism of tensor functors $(F,\gamma)(P_1,\chi_1)\to(G,\delta)(P_2,\chi_2)$ since the diagram
$$ 
\xymatrix{
\omega_\CC F P_1 \ar^-{\id_{\omega_\CC}\star\psi}[rr] \ar_-{\gamma\star\id_{P_1}}[d] & & \omega_\CC G P_2 \ar^-{\delta\star\id_{P_2}}[d] \\
\omega_\AA P_1 \ar_-{\chi_1}[rd] &  & \omega_\BB P_2 \ar^{\chi_2}[ld] \\
& \omega_{\AA\times_\CC\BB} &
}
$$
commutes by definition of $\chi_2$.

To state the universal property of the fibre product  $(\AA\times_\CC\BB,\omega_{\AA\times_\CC\BB})$ we need $2$-commutative diagrams of neutralized Tannakian categories.
Let $(F,\gamma)\colon (\AA,\omega_{\AA})\to (\CC,\omega_{\BB})$ and $(G,\delta)\colon (\BB,\omega_{\BB})\to(\CC,\omega_{\CC})$ be tensor functors between neutralized Tannakian categories. A $2$-commutative diagram (with respect to $(F,\gamma)$ and $(G,\delta)$) is a quadruple $((\DD,\omega_\DD), (Q_1,\epsilon_1), (Q_2,\epsilon_2), \alpha)$ consisting of a neutralized Tannakian category $(\DD,\omega_\DD)$, tensor functors $(Q_1,\epsilon_1)\colon (\DD,\omega_\DD)\to (\AA,\omega_\AA)$ and $(Q_2,\epsilon_2)\colon (\DD,\omega_\DD)\to (\BB,\omega_\BB)$ and an isomorphism $\alpha\colon (F,\gamma)(Q_1,\epsilon_1)\to (G,\delta)(Q_2,\epsilon_2)$ of tensor functors.

A morphism $$((H,\eta),\beta_1,\beta_2)\colon ((\DD,\omega_\DD), (Q_1,\epsilon_1), (Q_2,\epsilon_2), \alpha)\to ((\DD',\omega_{\DD'}), (Q'_1,\epsilon'_1), (Q'_2,\epsilon'_2), \alpha')$$ of $2$\=/commutative diagrams consists of a tensor functor $(H,\eta)\colon (\DD,\omega_{\DD})\to (\DD',\omega_{\DD'})$ and isomorphisms
$\beta_1\colon (Q_1,\epsilon_1)\to (Q'_1,\epsilon_1')(H,\eta)$ and $\beta_2\colon (Q_2,\epsilon_2)\to (Q'_2,\epsilon_2')(H,\eta)$ of tensor functors such that
$$
\xymatrix{
	FQ_1 \ar[r]^-{\id_F\star\beta_1} \ar_-{\alpha}[d] & FQ_1'H \ar^-{\alpha'\star\id_H}[d] \\
	GQ_2 \ar[r]^-{\id_G\star\beta_2} & GQ_2'H	
}
$$
commutes. Finally, an isomorphism between two morphisms $((H,\eta),\beta_1,\beta_2)$, $((H',\eta'),\beta'_1,\beta'_2)$ of $2$-commutative diagrams with the same source and target, is an isomorphism $\theta\colon (H,\eta)\to (H',\eta')$ of tensor functors  such that $$
\xymatrix{
	Q_1 \ar^-{\beta_1}[r] \ar_{\beta_1'}[rd] & Q_1'H \ar^-{\id_{Q_1'}\star\theta}[d] \\
	& Q_1'H'	
}
\quad \text{ and } \quad
\xymatrix{
	Q_2 \ar^-{\beta_2}[r] \ar_{\beta_2'}[rd] & Q_2'H \ar^-{\id_{Q_2'}\star\theta}[d] \\
	& Q_2'H'	
}
$$
commute.

%

\begin{prop} \label{prop: fibre product Tannakian categories}
		Let $(F,\gamma)\colon (\AA,\omega_{\AA})\to (\CC,\omega_{\BB})$ and $(G,\delta)\colon (\BB,\omega_{\BB})\to(\CC,\omega_{\CC})$ be tensor functors between neutralized Tannakian categories. 
	Then $((\AA\times_\CC\BB,\omega_{\AA\times_\CC\BB}), (P_1,\chi_1), (P_2, \chi_2), \psi)$ satisfies the following universal property. For every $2$-commutative diagram $((\DD,\omega_\DD), (Q_1,\epsilon_1), (Q_2,\epsilon_2), \alpha)$ of neutralized Tannakian categories, there exists a morphism 
		\begin{equation} \label{eq: long morph}
	((H,\eta),\beta_1,\beta_2)\colon ((\DD,\omega_\DD), (Q_1,\epsilon_1),(Q_2,\epsilon_2), \alpha)\to ((\AA\times_\CC\BB,\omega_{\AA\times_\CC\BB}), (P_1,\chi_1), (P_2, \chi_2), \psi)
	\end{equation} 
	of $2$-commutative diagrams. If $((H',\eta'),\beta'_1,\beta'_2)$ is another such morphism, then there exists an isomorphism between $((H,\eta),\beta_1,\beta_2)$ and $((H',\eta'),\beta'_1,\beta'_2)$.
\end{prop}
\begin{proof}
	The functor $H\colon \DD\to \AA\times_\CC\BB$ defined in Lemma \ref{lemma: universal property of fibre product} can be considered as a tensor functor with functorial isomorphism $H(-)\otimes H(-)\simeq H(-\otimes -)$ induced by the functorial ismorphisms $Q_i(-)\otimes Q_i(-)\simeq Q_i(-\otimes -)$ for $i=1,2$. Setting $\eta\colon \omega_{\AA\times_\CC\BB}H=\omega_\AA Q_1\xrightarrow{\epsilon_1}\omega_\DD$ we obtain a tensor functor $(H,\eta)\colon (\DD,\omega_\DD)\to (\AA\times_{\CC}\BB,\omega_{\AA\times_\CC\BB})$ between neutralized Tannakian categories. As $(P_1,\chi_1)(H,\eta)=(Q_1,\epsilon_1)$, the identity $\beta_1\colon (Q_1,\epsilon_1)\to (P_1,\chi_1)(H,\eta)$ is a morphism of tensor functors.
	By definition of $\chi_2$, the diagram
	$$\xymatrix{
	\omega_\CC FP_1 \ar^-{\id_{\omega_\CC}\star\psi}[r] \ar_-{\gamma\star\id_{P_1}}[d] & \omega_\CC GP_2 \ar^-{\delta\star\id_{P_2}}[d] \\
	\omega_\AA P_1 & \omega_\BB P_2 \ar_-{\chi_2}[l]
	}
	$$
	commutes. Therefore, the upper square in
	\begin{equation} \label{eq: diag}
		\xymatrix{
		\omega_\CC FP_1H \ar^-{\id_{\omega_\CC}\star\psi\star\id_H}[rr] \ar_-{\gamma\star\id_{P_1H}}[d]  &  & \omega_\CC GP_2H \ar^-{\delta\star\id_{P_2H}}[d] \\
		\omega_\AA P_1H \ar_-{\epsilon_1}[rd] &  & \omega_\BB P_2H \ar_-{\chi_2\star\id_H}[ll] \ar^-{\epsilon_2}[ld] \\
		& \omega_\DD&			
	}
	\end{equation}
	commutes. Noting that $P_1H=Q_1$, $P_2H=Q_2$ and $\psi\star\id_H=\alpha$, we see that the outer diagram in (\ref{eq: diag}) commutes because $\alpha\colon (F,\gamma)(Q_1,\epsilon_1)\to (G,\delta)(Q_2,\epsilon_2)$ is an isomorphism of tensor functors. Thus also the lower triangle in (\ref{eq: diag}) commutes. This shows that $(P_2,\chi_2)(H,\eta)=(Q_2,\epsilon_2)$. Thus the identity $\beta_2\colon (Q_2,\epsilon_2)\to (P_2,\chi_2)(H,\eta)$ is an isomorphism of tensor functors. We have thus constructed the morphism (\ref{eq: long morph}).
	
	If $((H',\eta'),\beta_1',\beta_2')$ is another such morphism, the isomorphism $\theta\colon H\to H'$ constructed in Lemma~\ref{lemma: universal property of fibre product} is a tensor isomorphism because $\beta_1'\colon Q_1\to P_1H'$ and $\beta_2'\colon Q_2\to P_2H'$ are tensor isomorphisms. 
	Moreover, the fact that $\beta_1'\colon(Q_1,\epsilon_1)\to (P_1,\chi_1)(H',\eta')$ is an isomorphism of tensor functors, implies that 
	$$
	\xymatrix{
	\omega_{\AA\times_\CC\BB}H=\omega_\AA Q_1 \ar^-{\id_{\omega_{\AA\times_\CC\BB}}\star\theta=\id_{\omega_\AA}\star \beta_1'}[rr] \ar_-{\eta=\epsilon_1}[rd] & & \omega_\AA P_1H'=\omega_{\AA\times_\CC\BB}H'  \ar^-{\eta'}[ld] \\
	& \omega_\DD &
	}
		$$
		commutes. So $\theta\colon (H,\eta)\to (H',\eta')$ is an isomorphism of tensor functors.	
\end{proof}

\begin{cor} \label{t1-11}
	Let $(\AA_1,\omega_{\AA_1})$, $(\AA_2,\omega_{\AA_2})$ and $(\BB,\omega_{\BB})$ be neutralized Tannakian categories with fundamental groups $G_1, G_2$ and $H$. For tensor functors $(\AA_1,\omega_{\AA_1})\to (\BB,\omega_{\BB})$ and $(\AA_2,\omega_{\AA_2})\to(\BB,\omega_{\BB})$, the fundamental group of the neutralized Tannakian category $(\AA_1\times_\BB \AA_2,\omega_{\AA_1\times_\BB\AA_2})$ is the amalgamated free product $G_1*_H G_2$.
\end{cor}
\begin{proof}
	Let $G$ denote the fundamental group of $(\AA_1\times_\BB \AA_2,\omega_{\AA_1\times_\BB\AA_2})$.
	The $2$-commutative diagram 
	\begin{equation} \label{eq: diagram fibre product}
	\xymatrix{
		& (\AA_1\times_{\BB}\AA_2, \omega_{\AA_1\times_\BB\AA_2}) \ar_-{(P_1,\chi_1)}[ld] \ar^-{(P_2,\chi_2)}[rd] & \\
		(\AA_1,\omega_{\AA_1}) \ar[rd] & & (\AA_2,\omega_{\AA_2}) \ar[ld] \\	
		& (\BB,\omega_{\BB})& 
	}	
	\end{equation}
	of neutralized Tannakian categories dualizes to a commutative diagram
	$$
	\xymatrix{
	& G & \\
	G_1 \ar^-{i_1}[ru] & & G_2 \ar_-{i_2}[lu] \\
	& \ar[lu] H \ar[ru] &
}
	$$
	of $k$-groups. Let
		\begin{equation}
		\label{eq: diagram for groups}
	\xymatrix{
		& G' & \\
		G_1 \ar^-{i'_1}[ru] & & G_2 \ar_-{i'_2}[lu] \\
		& \ar[lu] H \ar[ru] &
	}
	\end{equation}
be a commutative diagram of $k$-groups. As the morphisms between fundamental groups are in bijection with isomorphism classes of tensor functors between neutralized Tannakian categories (Lemma \ref{lemma: morphisms of groups and isom classes of functors}), there exists a $2$-commutative diagram
\begin{equation} \label{eq: diagram cat}
\xymatrix{
& (\Rep(G'),\omega_{G'}) \ar_-{(Q_1,\epsilon_1)}[ld] \ar^-{(Q_2,\epsilon_2)}[rd] & \\
  (\AA_1,\omega_{\AA_1}) \ar[rd] & & (\AA_2,\omega_{\AA_2}) \ar[ld] \\	
 & (\BB,\omega_{\BB})& 	
}
\end{equation}
of neutralized Tannakian categories that dualizes to (\ref{eq: diagram for groups}). By Proposition \ref{prop: fibre product Tannakian categories} there exists a morphism 
$((H,\eta),\beta_1,\beta_2)$ of $2$-commutative diagrams from (\ref{eq: diagram cat}) to (\ref{eq: diagram fibre product}). The morphism $(H,\eta)\colon (\Rep(G'),\omega_{G'})\to (\AA_1\times_\BB \AA_2,\omega_{\AA_1\times_\BB\AA_2})$ dualizes to a morphism $\f\colon G\to G'$ of $k$-groups. The $2$-commutativity of 
$$ 
\xymatrix{
(\Rep(G'),\omega_{G'}) \ar^-{(H,\eta)}[rr] \ar_-{(Q_j,\epsilon_j)}[rd] & & (\AA_1\times_\BB \AA_2,\omega_{\AA_1\times_\BB\AA_2}) \ar^-{(P_j,\chi_j)}[ld] 	\\
& (\AA_j,\omega_{\AA_j}) &	
}
$$
realized by $\beta_j$ implies that 
\begin{equation} \label{eq: diagram groups left right}
\xymatrix{
G \ar^-{\f}[rr] & & G' \\
& \ar_-{i_j'}[ru] \ar^-{i_j}[lu] G_j &
}
\end{equation}
commutes for $j=1,2$. Finally, the uniqueness of $\f$ (making (\ref{eq: diagram groups left right}) commute for $j=1,2$) follows from the fact that $((H,\eta),\beta_1,\beta_2)$ is unique up to an isomorphism.
\end{proof}

\begin{cor} \label{t1-12}
	Let $H\to G_1$ and $H\to G_2$ be morphisms of $k$-groups. Then the amalgamated free product $G_1*_H G_2$ exists.
\end{cor}
\begin{proof}
	The restriction functors $\Rep(G_i)\to \Rep(H)$ ($i=1,2$) can be regarded as tensor functors $(\Rep(G_i),\omega_{G_i})\to (\Rep(H),\omega_H)$ between neutralized Tannakian categories. By Corollary \ref{t1-11} the fundamental group of $\Rep(G_1)\times_{\Rep(H)}\Rep(G_2)$ is $G_1*_H G_2$.
\end{proof}

\section{The Hopf algebra of the amalgamated free product} \label{sec:2}
In this section we offer an explicit description of the Hopf algebra $k[G_1*_HG_2]$ of the amalgamated free product of two 
morphisms $H\to G_1$ and $H\to G_2$ of $k$-groups in terms of the Hopf algebras $k[G_1]$, $k[G_2]$ and $k[H]$. In other words, we explicitly describe the fibre product in the category of Hopf algebras. 
(Recall that all Hopf algebras are assumed to be commutative.) We also discuss properties of amalgamated free products, such as being connected or geometrically reduced. 


\subsection{Fibre products of Hopf algebras}
To motivate our explicit construction of the fibre product of Hopf algebras, let us revisit the construction of the amalgamated free product in the category of groups from a more categorical/conceptual perspective. For now, assume that $\f_1\colon H\to G_1$ and $\f_2\colon H\to G_2$ are morphisms of abstract groups. The amalgamated free product $G_1*_HG_2$ (in the category of groups) can be constructed as follows: First form the disjoint union $G_1\uplus G_2$ and then build the free monoid $M(G_1\uplus G_2)$ on the set $G_1\uplus G_2$. Finally, the amalgamated free product $G_1*_HG_2$ is obtained from $M(G_1\uplus G_2)$, by quotienting out by the relations $$g_1\cdot g_1'=g_1g_1' \quad \forall \ g_1, g_1'\in G_1,$$  $$g_2\cdot g_2'=g_2g_2' \quad \forall \ g_2, g_2'\in G_2,$$
$$1_{G_1}=1_{M(G_1\uplus G_2)},$$
$$1_{G_2}=1_{M(G_1\uplus G_2)},$$  
$$\f_1(h)=\f_2(h) \quad \forall \ h\in H.$$
Here the $\cdot$ is used to denote multiplication in $M(G_1\uplus G_2)$.

This construction can be dualized. Let $k[G_1]\to k[H]$ and $k[G_2]\to k[H]$ be morphisms of Hopf algebras. The disjoint union $G_1\uplus G_2$ corresponds to the product algebra $k[G_1]\oplus k[G_2]=k[G_1]\times k[G_2]$. The free monoid is the left adjoint of the forgetful functor from the category of monoids to the category of sets. In the Hopf algebra world, we should thus consider the right adjoint $A\rightsquigarrow C(A)$ of the forgetful functor from the category of bialgebras to the category of algebras. Finally, taking the quotient of $M(G_1\uplus G_2)$, corresponds to choosing an appropriate subbialgebra $k[G_1*_HG_2]$ of $C(k[G_1]\times k[G_2])$ (that turns out to be a Hopf algebra). 

To carry out the above idea, we need an explicit description of the right adjoint $A\rightsquigarrow C(A)$ of the forgetful functor from the category of bialgebras to the category of algebras. To this end, let us recall the construction of the \emph{cofree coalgebra} $C(V)$ on a $k$-vector space $V$. See \cite[Section 6.4]{Sweedler:HopfAlgebras} or \cite{Hazewinkel:CofreeCoalgebras}. It comes with a $k$-linear map $\pi\colon C(V)\to V$ characterized by the following universal property: If $D$ is a coalgebra and $\varphi\colon D\to V$ is a $k$-linear map, there exists a unique morphism $\tilde{\varphi}\colon D\to C(V)$ of coalgebras making
$$
\xymatrix{
	C(V) \ar^-\pi[rr] & & V \\
	& D \ar_-\varphi[ru] \ar^-{\tilde{\varphi}}@{..>}[lu]	&
}
$$
commutative. In other words, the functor $V\rightsquigarrow C(V)$ is right adjoint to the forgetful functor from coalgebras to vector spaces. We consider the tensor algebra $TV=k\oplus V\oplus V^{\otimes 2}\oplus\cdots$. However, we disregard the usual (non-commutative) algebra structure on $TV$. Instead, we define a coalgebra structure on $TV$.
For every $n\in\nn$ and $0\leq i,j$ with $i+j=n$, we have an isomorphism $\psi_{i,j}\colon V^{\otimes n}\to V^{\otimes i}\otimes V^{\otimes j}$. Here $V^{0\otimes}=k$ and $V^{1\otimes}=V$.
For $t\in TV$ homogeneous of degree $n$, i.e., $t\in V^{\otimes n}\subseteq TV$, we define $$\Delta(t)=\sum_{i=0}^n\psi_{i,n-i}(t)\in\bigoplus_{i=0}^n V^{\otimes i}\otimes V^{\otimes (n-i)}\subseteq TV\otimes TV.$$
By linearity, this assignment extends to a $k$-linear map  
$\Delta\colon TV\to TV\otimes TV$. Define $\epsilon\colon TV\to k$ as the projection onto the degree zero component. It is straight forward to check that $(TV, \Delta, \epsilon)$ is a coalgebra. However, this is not yet the cofree coalgebra on $V$. To construct $C(V)$ we need to consider the completion $\hat{T}V=k\times V\times V^{\otimes  2}\times\ldots$ of $TV$. Thinking of the elements of $\hat{T}V$ as infinite sums, we have an inclusion $TV\subseteq\hat{T}V$.

One cannot extend $\Delta\colon TV\to TV\otimes TV$ directly to a map $\Delta\colon \hat{T}V\to \hat{T}V\otimes\hat{T}V$, because the natural extension of $\Delta\colon TV\to TV\otimes TV$ is a map $\Delta\colon\hat{T}V\to \hat{T}V\hat{\otimes}\hat{T}V$, where  $\hat{T}V\hat{\otimes}\hat{T}V=\prod_{i,j\in\nn}V^{\otimes i}\otimes V^{\otimes j}$. Note that $\hat{T}V\otimes\hat{T}V$ is a proper subset of $\hat{T}V\hat{\otimes}\hat{T}V$: If we think of an element of $\hat{T}V\hat{\otimes}\hat{T}V$ as an infinite matrix
\begin{equation} \label{eq: matrix notation}
\begin{pmatrix}
f_{00} & f_{01} & f_{02} & \ldots \\
f_{10} & f_{11} & f_{12} & \ldots \\
\vdots & \vdots & &
\end{pmatrix}
\end{equation}
with $f_{ij}\in V^{\otimes i}\otimes V^{\otimes j}$, then the elements of
$\hat{T}V\otimes\hat{T}V$ are finite sums of elements of the form

$$f\otimes g=
\begin{pmatrix}
f_0\otimes g_0 & f_0\otimes g_1 & f_0\otimes g_2 & \ldots \\
f_1\otimes g_0 & f_1\otimes g_1 & f_1\otimes g_2 & \ldots \\
\vdots & \vdots & &
\end{pmatrix}
$$
with $f=(f_0,f_1,\ldots),\ g=(g_0,g_1,\ldots)\in \hat{T}V$. Note that with respect to this matrix representation, the map $\Delta\colon \hat{T}V\to \hat{T}V\hat{\otimes}\hat{T}V$ is given by
$$\Delta(f)=\begin{pmatrix}
\psi_{00}(f_0) & \psi_{01}(f_1) & \psi_{02}(f_2) & \ldots \\
\psi_{10}(f_1) & \psi_{11}(f_2) &  \\
\psi_{21}(f_2) &  & & \\
\vdots & & & \ddots
\end{pmatrix} $$
for $f=(f_0,f_1,\ldots)\in \hat{T}V$. Set
$$C(V)=\{f\in\hat{T}V|\ \Delta(f)\in \hat{T}V\otimes \hat{T}V\subseteq \hat{T}V\hat{\otimes} \hat{T}V \}.$$
Following \cite{Hazewinkel:CofreeCoalgebras}, we call the elements of $\hat{T}V$ that belong to $C(V)$ \emph{representative}.

One can show (\cite[Theorem 3.12]{Hazewinkel:CofreeCoalgebras}), that $\Delta(C(V))\subseteq C(V)\otimes C(V)\subseteq \hat{T}V\otimes \hat{T}V$ and so $C(V)$ (together with the induced structure maps) is a coalgebra. Let $\pi\colon C(V)\to V$ be the projection to the degree one component. Then $\pi$ satisfies the universal property of the cofree coalgebra (\cite[Theorem 3.14]{Hazewinkel:CofreeCoalgebras}). Indeed, if $(D,\mu,\varepsilon)$ is a coalgebra with a $k$-linear map $\varphi\colon D\to V$, the unique $\tilde{\varphi}\colon D\to C(V)$ such that $\pi\tilde{\varphi}=\varphi$ is given by $\tilde{\varphi}(d)=(\varepsilon(d),\varphi(d),\varphi^{\otimes 2}(\mu_2(d)),\varphi^{\otimes 3}(\mu_3(d)),\ldots)\in C(V)$ for $d\in D$, where
$\mu_i\colon D\to D^{\otimes i}$ is the $(i-1)$-st iteration of $\mu$. 

If $A$ is a $k$-algebra, then $C(A)$ naturally acquires the structure of a $k$-algebra (Cf. \cite[p.~134]{Sweedler:HopfAlgebras}). Indeed, this follows from the universal property: Let $m\colon A\otimes A\to A$ be the multiplication. As $C(A)\otimes C(A)$ has the structure of a coalgebra (the product coalgebra), the map $C(A)\otimes C(A)\xrightarrow{\pi\otimes \pi}A\otimes A\xrightarrow{m} A$ lifts to a multiplication map $C(A)\otimes C(A)\to C(A)$. Similarly, as $k$ has the trivial coalgebra structure, the unit $k\to A$ lifts to a unit $k\to C(A)$.
Explicitly, the unit $k\to C(A)$ is the map $k\to C(A),\ \lambda\mapsto (\lambda,\lambda,\lambda,\ldots)$ and multiplication in $C(A)$ is componentwise, using the multiplication in $A^{\otimes i}$. One sees that $C(A)$ is a bialgebra and $\pi\colon C(A)\to A$ is a morphism of algebras. Moreover, for every bialgebra $B$ with an algebra morphism $\varphi\colon B\to A$, the unique $\tilde{\varphi}\colon B\to C(A)$ such that $\tilde{\varphi}\pi=\varphi$ is a morphism of bialgebras. In other words, $A\rightsquigarrow C(A)$ is the sought for right adjoint of the forgetful functor from bialgebras to algebras. 

\medskip

Let $\alpha_1\colon k[G_1]\to k[H]$ and $\alpha_2\colon k[G_2]\to k[H]$ be morphisms of Hopf algebras. We will construct their fibre product $k[G_1]\times_{k[H]}k[G_2]=k[G_1*_HG_2]$ as a suitable subalgebra of $C(k[G_1]\oplus k[G_2])$.

Note that for $n\geq 0$ the vector space $(k[G_1]\oplus k[G_2])^{\otimes n}$ is the direct sum of the subspaces 
$$T_{i_1,\ldots,i_n}=k[G_{i_1}]\otimes \ldots\otimes k[G_{i_n}]$$ with $i_1,\ldots,i_n\in\{1,2\}$.
For $\ell\in \{1,\ldots,n\}$ define a map 
$$\psi_{i_1,\ldots,i_n}^\ell\colon T_{i_1,\ldots,i_n}\to T_{i_1,\ldots,i_{\ell-1},i_\ell,i_\ell,i_{\ell+1},\ldots,i_n}$$ by applying the comultiplication $\Delta_{i_l}\colon k[G_{i_\ell}]\to k[G_{i_\ell}]\otimes k[G_{i_\ell}]$ on the $\ell$-th factor and the identity on all the other factors. Similarly, define
$$\delta_{i_1,\ldots,i_n}^\ell\colon T_{i_1,\ldots,i_n}\to T_{i_1,\ldots,i_{\ell-1},i_{\ell+1},\ldots,i_n}$$ by applying the counit $\epsilon_{i_\ell}\colon k[G_{i_\ell}]\to k$ on the $\ell$-th factor and the identity on all the other factors.
We also define 
$$\theta^\ell_{i_1,\ldots,i_n}\colon T_{i_1,\ldots,i_n}\to k[G_{i_1}]\otimes\ldots\otimes k[G_{i_{\ell-1}}]\otimes k[H]\otimes k[G_{i_\ell}]\otimes\ldots\otimes k[G_{i_n}] $$
by applying $\alpha_{i_\ell}\colon k[G_{i_\ell}]\to k[H]$ on the $\ell$-th factor and the identity on all the other factors.

An element $f\in \hat{T}(k[G_1]\oplus k[G_2])$ can uniquely be written in the form 
\begin{equation} \label{eq: parts of f}
f=\sum_{n=0}^\infty\sum_{(i_1,\ldots,i_n)\in \{1,2\}^n} f_{i_1,\ldots,i_n}
\end{equation}
with $f_{i_1,\ldots,i_n}\in T_{i_1,\ldots,i_n}$. Note that for $n=0$, $T_{i_1,\ldots,i_n}=k$ and in (\ref{eq: parts of f}), the inner sum for $n=0$ has one summand $f_\emptyset\in k$.

We say that $f$ is \emph{multiplicatively coherent} if 
\begin{equation} \label{eq: comultiplication condition}
\psi^\ell_{i_1,\ldots,i_n}(f_{i_1,\ldots,i_n})=f_{i_1,\ldots,i_\ell,i_\ell,\ldots,i_n}
\end{equation}
for all $n\geq 1$, $i_1,\ldots,i_n\in\{1,2\}$ and $\ell\in \{1,\ldots,n\}$.
If 
\begin{equation} \label{eq: counit condition}
\delta^\ell_{i_1,\ldots,i_n}(f_{i_1,\ldots,i_n})=f_{i_1,\ldots,i_{\ell-1},i_{\ell+1},\ldots,i_n}
\end{equation}
for all $n\geq 1$, $i_1,\ldots,i_n\in\{1,2\}$ and $\ell\in \{1,\ldots,n\}$, then $f$ is \emph{unitaly coherent}.
If 
\begin{equation} \label{eq: H coherence}
\theta^\ell_{i_1,\ldots,i_{\ell-1},1,i_{\ell+1},\ldots,i_n}(f_{i_1,\ldots,i_{\ell-1},1,i_{\ell+1},\ldots,i_n})=\theta^\ell_{i_1,\ldots,i_{\ell-1},2,i_{\ell+1},\ldots,i_n}(f_{i_1,\ldots,i_{\ell-1},2,i_{\ell+1},\ldots,i_n})
\end{equation}
for all $n\geq 1$, all $\ell\in\{1,\ldots,n\}$ and $i_1,\ldots,i_{\ell-1},i_{\ell+1},\ldots,i_n\in\{1,2\}$, then 
$f$ is $H$-coherent.
We call $f$ \emph{coherent} if it is multiplicatively, unitaly and $H$-coherent.

If we think of $f$ as a function on the words in $G_1\uplus G_2$, multiplicative coherence means that $f$ takes the same value on words that only differ by replacing two adjacent elements from the same group with their product. Unital coherence means that $f$ takes the same value on words that only differ by deleting the identity element of $G_1$ or the identity element of $G_2$. The meaning of $H$-coherence is that if $g_1\in G_1$ and $g_2\in G_2$ come from the same element in $H$, then the value of $f$ on a word does not change if a letter $g_1$ is replaced by $g_2$.
We define
$$k[G_1*_HG_2]=\{f\in C(k[G_1]\oplus k[G_2])|\ f \text{ is coherent}\}.$$

As the $\psi^\ell_{i_1,\ldots,i_n}$'s, $\delta^\ell_{i_1,\ldots,i_n}$'s and $\theta_{i_1,\ldots,i_n}^\ell$'s are morphisms of $k$-algebras, it is clear that $k[G_1*_HG_2]$ is a subalgebra of $C(k[G_1]\oplus k[G_2])$. On the other hand, it is not so clear that  $k[G_1*_HG_2]$ is a subcoalgebra of $C(k[G_1]\oplus k[G_2])$. To see this, we will work with a different characterization of representative elements based on duality.

We briefly go back to the general setup where $V$ is a vector space over $k$. We denote with $V^*=\Hom_k(V,k)$ the dual vector space of functionals on $V$.
Every element $f$ of $\hat{T}V$ defines a functional $\fun(f)\in (TV^*)^*$ on $TV^*$. In detail, for every $n\geq 0$ we have an evaluation map $$V^{\otimes n}\times (V^*)^{\otimes n}\to k,\ (v,a)\mapsto a(v) $$ determined by $(v_1\otimes\ldots\otimes v_n, a_1\otimes\ldots\otimes a_n)\mapsto a_1(v_1)\ldots a_n(v_n)$.
Then for $f=(f_0,f_1,\ldots)\in \hat{T}V$ and $a=a_0+a_1+\ldots+a_r\in TV^*$ one sets $\fun(f)(a)=\sum_{i=0}^ra_i(f_i)$. The map $\hat{T}V\to (TV^*)^*,\ f\mapsto \fun(f)$ is injective. In a similar fashion, any $f\in \hat{T}V\hat{\otimes}\hat{T}V=\prod_{i,j\in\nn}V^{\otimes i}\otimes V^{\otimes j}$ defines a functional $\fun(f)$ on $TV^*\otimes TV^*$. The ``comultiplication'' $\Delta$ on $\hat{T}V$ is dual to the usual (noncommutative) multiplication in the tensor algebra $TV^*$ in the sense that
\begin{equation} \label{eq: duality}
\fun(\Delta(f))(a\otimes b)=\fun(f)(ab)
\end{equation}
for $f\in \hat{T}V$ and $a,b\in TV^*$.

For a functional $F$ on $TV^*$ and $b\in TV^*$ we define a functional $R_bF$ on $TV^*$ by $(R_bF)(a)=F(ab)$ for all $a\in TV^*$. Similarly, we define $L_aF$ by $(L_aF)(b)=F(ab)$. We abbreviate $R_bf=R_b\fun(f)$ and $L_af=L_a\fun(f)$. Note that 
$$Rf=\{R_bf|\ b\in TV^*\} \text{ and } Lf=\{L_af|\ a\in TV^*\}$$
are subspaces of $(TV^*)^*$. The following lemma is essentially Theorem 3.12 in \cite{Hazewinkel:CofreeCoalgebras}.

\begin{lemma} \label{lemma: Hazewinkel}
	For $f\in \hat{T}V$ the following are equivalent:
	\begin{enumerate}
		\item The element $f$ is representative.
		\item The vector space $Rf$ is finite dimensional.
		\item The vector space $Lf$ is finite dimensional.
	\end{enumerate}
	Moreover, if these conditions are satisfied, then $\Delta(f)$ can be written in the form $\Delta(f)=\sum_{i=1}^rg_i\otimes h_i$ with $g_i,h_i\in \hat{T}V$ representative and $\fun(g_i)\in Rf$ and $\fun(h_i)\in Lf$ for $i=1,\ldots,r$.
\end{lemma}
\begin{proof}
	Assume that $f$ is representative. So $\Delta(f)=\sum_{i=1}^rg_i\otimes h_i$ for some $g_i,h_i\in \hat{T}V$. Using (\ref{eq: duality}) we find 
	$$L_af(b)=R_bf(a)=\fun(f)(ab)=\sum_{i=0}^r\fun(g_i)(a)\fun(h_i)(b)$$
	for $a,b\in TV^*$. Thus $R_bf$ lies in the subspace generated by $\fun(g_1),\ldots,\fun(g_r)$ and $L_af$ lies in the subspace generated by $\fun(h_1),\ldots,\fun(h_r)$. Thus (i) implies (ii) and (iii). The implication (ii)$\Rightarrow$(i) is shown in \cite[Theorem 3.12]{Hazewinkel:CofreeCoalgebras} and the implication (iii)$\Rightarrow$(i) works similar to (ii)$\Rightarrow$(i). The final statement of the lemma is contained in the proof of \cite[Theorem 3.12]{Hazewinkel:CofreeCoalgebras}. See \cite[Remark 3.24]{Hazewinkel:CofreeCoalgebras}.
\end{proof}

We return to the discussion of the subalgebra $k[G_1*_HG_2]$ of $C(k[G_1]\oplus k[G_2])$.
For $n\geq 0$, the vector space $(k[G_1]^*\oplus k[G_2]^*)^{\otimes n}$ is the direct sum of the subspaces 
$$T^*_{i_1,\ldots,i_n}=k[G_{i_1}]^*\otimes \ldots\otimes k[G_{i_n}]^*$$ with $i_1,\ldots,i_n\in\{1,2\}$.

The comultiplication $\Delta_i\colon k[G_i]\to k[G_i]\otimes k[G_i]$ dualizes to a map $(k[G_i]\otimes k[G_i])^*\to k[G_i]^*$. Precomposing with $k[G_i]^*\otimes k[G_i]^*\to (k[G_i]\otimes k[G_i])^*$, thus yields a map $\Delta_i^*\colon k[G_i]^*\otimes k[G_i]^*\to k[G_i]^*$.

For $i_1,\ldots,i_n\in\{1,2\}$ and $\ell\in \{1,\ldots,n\}$ we define a map $$^*\psi^\ell_{i_1,\ldots,i_n}\colon T^*_{i_1,\ldots,i_\ell,i_\ell,\ldots,i_n}\to T^*_{i_1,\ldots,i_n}$$ by applying $\Delta_{i_\ell}^*$ on $k[G_{i_\ell}]^*\otimes k[G_{i_\ell}]^*$ and the identity on all the other factors. Similarly, we define a map
$$^*\delta^\ell_{i_1,\ldots,i_n}\colon T^*_{i_1,\ldots,i_{\ell-1},i_{\ell+1},\ldots,i_n}\to T^*_{i_1,\ldots,i_n}$$
by sending $a_1\otimes\ldots\otimes a_{n-1}$ to $a_1\otimes\ldots\otimes a_{\ell-1}\otimes \epsilon_{i_\ell}\otimes a_\ell,\ldots,a_{n-1}$.
We also define 
$$^*\theta^\ell_{i_1,\ldots,i_n}\colon k[G_{i_1}]^*\otimes \ldots\otimes k[G_{i_{\ell-1}}]^*\otimes k[H]^*\otimes k[G_{i_{\ell+1}}]^*\otimes\ldots\otimes k[G_{i_n}]^*\longrightarrow T^*_{i_1,\ldots,i_n} $$
by applying $\alpha_{i_\ell}^*\colon k[H]^*\to k[G_{i_\ell}]^*$ on the $\ell$-th factor and the identity on all the other factors.

A functional $F\in T(k[G_1]^*\oplus k[G_2]^*)^*$ is called \emph{multiplicatively coherent} if 
$$
\xymatrix{
	T^*_{i_1,\ldots,i_\ell,i_\ell,\ldots,i_n} \ar^-{^*\psi^\ell_{i_1,\ldots,i_n}}[rr] \ar_-F[rd] & & T^*_{i_1,\ldots,i_n} \ar^-F[ld] \\
	& 	k &
} 	
$$ 	
commutes for every $^*\psi^\ell_{i_1,\ldots,i_n}$. If 
$$
\xymatrix{
	T^*_{i_1,\ldots,i_{\ell-1},i_{\ell+1},\ldots,i_n} \ar_-F[rd] \ar^-{^*\delta^\ell_{i_1,\ldots,i_n}}[rr] & & T^*_{i_1,\ldots,i_n} \ar^-F[ld]\\
	& k &
}
$$  	
commutes for every $^*\delta^\ell_{i_1,\ldots,i_n}$, then $F$ is \emph{unitaly coherent}.
If
$$
\xymatrix{
	&  k[G_{i_1}]^*\otimes \ldots\otimes k[G_{i_{\ell-1}}]^*\otimes k[H]^*\otimes k[G_{i_{\ell+1}}]^*\otimes\ldots\otimes k[G_{i_n}]^* \ar^-{^*\theta^\ell_{i_1,\ldots,i_{\ell-1},2,i_{\ell+1},\ldots,i_n}}[rd] \ar_-{^*\theta^\ell_{i_1,\ldots,i_{\ell-1},1,i_{\ell+1},\ldots,i_n}}[ld] & \\
T^*_{i_1\ldots, i_{\ell-1},1,i_{\ell+1},\ldots,i_n} \ar_-F[rd]	& & T^*_{i_1\ldots, i_{\ell-1},2,i_{\ell+1},\ldots,i_n} \ar^-F[ld] \\
	& k &  
}
$$
commutes for all $n\geq 1$, $\ell\in\{1,\ldots,n\}$ and $i_1,\ldots,i_{\ell-1},i_{\ell+1},\ldots,i_n\in\{1,2\}$, then $F$ is $H$-coherent. Finally, $F$ is \emph{coherent} if it is multiplicatively, unitaly and $H$\=/coherent.

The following lemma shows that whether or not $f\in\hat{T}(k[G_1]\oplus k[G_2])$ is coherent can be tested on $\fun(f)$.

\begin{lemma} \label{lemma: coherence dualizes}
	Let $f\in \hat{T}(k[G_1]\oplus k[G_2])$. Then $f$ is coherent if and only if $\fun(f)$ is coherent. 
\end{lemma} 
\begin{proof} 
	We first treat the case of multiplicative coherence.
	Fix $n\geq 1$, ${i_1,\ldots,i_n}\in\{1,2\}^n$ and $\ell\in\{1,\ldots,n\}$.	
	Let $a_1\otimes\ldots\otimes a_{n+1}\in T^*_{i_1,\ldots,i_\ell,i_\ell,\ldots,i_n}$.
	As $\Delta_{i_\ell}^*(a_\ell\otimes a_{\ell+1})(h)=(a_\ell\otimes a_{\ell+1})(\Delta_{i_\ell}(h))$ for $h\in k[G_{i_\ell}]$, we see that
	\begin{align*}
	^*\psi^\ell_{i_1,\ldots,i_n}(a_1&\otimes\ldots\otimes a_{n+1})(h_1\otimes\ldots\otimes h_n)= \\
	&=(a_1\otimes\ldots\otimes a_{\ell-1}\otimes\Delta_{i_\ell}^*(a_\ell\otimes a_{\ell+1})\otimes a_{\ell+2}\otimes\ldots\otimes a_{n+1})(h_1\otimes\ldots\otimes h_n)= \\
	&=(a_1\otimes\ldots\otimes a_{n+1})(h_1\otimes\ldots\otimes \Delta_{i_\ell}(h_\ell)\otimes\ldots\otimes h_n)= \\
	&=(a_1\otimes\ldots\otimes a_{n+1})(\psi^\ell_{i_1,\ldots,i_n}(h_1\otimes\ldots\otimes h_n))
	\end{align*}
	for $h_1\otimes\ldots\otimes h_n\in T_{i_1,\ldots,i_n}$. So $^*\psi^\ell_{i_1,\ldots,i_n}(a)(h)=a(\psi^\ell_{i_1,\ldots,i_n}(h))$ for all $a\in T^*_{i_1,\ldots,i_\ell,i_\ell,\ldots,i_n}$ and $h\in T_{i_1,\ldots,i_n}$.

	Assume that $f$ is multiplicatively coherent. Then, writing $f$ as is in (\ref{eq: parts of f}), we have $^*\psi^\ell_{i_1,\ldots,i_n}(a)(f_{i_1,\ldots,i_n})=a(f_{i_1,\ldots,i_\ell,i_\ell,\ldots,i_n})$ for all $a\in 	T^*_{i_1,\ldots,i_\ell,i_\ell,\ldots,i_n}$. Thus the diagram
	$$
	\xymatrix{
		T^*_{i_1,\ldots,i_\ell,i_\ell,\ldots,i_n} \ar^-{^*\psi^\ell_{i_1,\ldots,i_n}}[rr] \ar_-{\fun(f)}[rd] & & T^*_{i_1,\ldots,i_n} \ar^-{\fun(f)}[ld] \\
		& 	k &
	} 	
	$$ commutes and $\fun(f)$ is multiplicatively coherent.

	Conversely, assume that $\fun(f)$ is multiplicatively coherent.
	Then 
	$$a(\psi^\ell_{i_1,\ldots,i_n}(f_{i_1,\ldots,i_n}))={^*\psi^\ell_{i_1,\ldots,i_n}}(a)(f_{i_1,\ldots,i_n})=a(f_{i_1,\ldots,i_\ell,i_\ell,\ldots,i_n})$$
	for all $a\in T^*_{i_1,\ldots,i_\ell,i_\ell,\ldots,i_n}$ and so $\psi^\ell_{i_1,\ldots,i_n}(f_{i_1,\ldots,i_n})=f_{i_1,\ldots,i_\ell,i_\ell,\ldots,i_n}$, i.e., $f$ is multiplicatively coherent.
	
	The case of unital coherence is similar, using that $^*\delta^\ell_{i_1,\ldots,i_n}(a)(h)=a(\delta^\ell_{i_1,\ldots,i_n}(h))$ for $a\in T^*_{i_1,\ldots,i_{\ell-1},i_{\ell+1},\ldots,i_n}$ and $h\in T_{i_1,\ldots,i_n}$.	
	
	For the case of $H$-coherence one uses that
	$${^*\theta}^\ell_{i_1,\ldots,i_{\ell-1},i,i_{\ell+1},\ldots,i_n}(a)(h)=a(\theta^\ell_{i_1,\ldots,i_{\ell-1},i,i_{\ell+1},\ldots,i_n}(h))$$
	for  $$a\in  k[G_{i_1}]^*\otimes \ldots\otimes k[G_{i_{\ell-1}}]^*\otimes k[H]^*\otimes k[G_{i_{\ell+1}}]^*\otimes\ldots\otimes k[G_{i_n}]^*,$$
	$h\in T_{i_1,\ldots,i_{\ell-1},i,i_{\ell+1},\ldots,i_n} $
	and $i=1,2$.
	%
\end{proof}

%
%

The following lemma is the main reason, why the dual perspective is helpful for showing that $k[G_1*_HG_2]$ is a subcoalgebra of $C(k[G_1]\oplus k[G_2])$. 

\begin{lemma} \label{lemma; coherence preserved under translation}
	Let $F$ be a coherent functional on $T(k[G_1]^*\oplus k[G_2]^*)$. Then $R_bF$ is coherent for every $b\in T(k[G_1]^*\oplus k[G_2]^*)$. Similarly for $L_aF$.
\end{lemma}
\begin{proof}
	We first treat the case of multiplicative coherence.
	As the sum of multiplicatively coherent functionals is multiplicatively coherent and $R_{b_1+b_2}F=R_{b_1}F+R_{b_2}F$, we can reduce to the case that $b=b_1\otimes\ldots\otimes b_n$ is a pure tensor in some $T^*_{i_1,\ldots,i_n}$. Moreover, as $R_{b_1b_2}F=R_{b_1}(R_{b_2}F)$, we are reduced to the case $b\in k[G_i]^*$ for some $i\in\{1,2\}$.
	Consider the diagram	
	$$
	\xymatrix{
		T^*_{i_1,\ldots,i_\ell,i_\ell,\ldots,i_n} \ar^{^*\psi^\ell_{i_1,\ldots,i_n}}[rr] \ar[d] & & T^*_{i_1,\ldots,i_n} \ar[d] \\
		T^*_{i_1,\ldots,i_\ell,i_\ell,\ldots,i_n}\otimes k[G_i]^* \ar^-\chi[rr] \ar_-F[rd] & & T^*_{i_1,\ldots,i_n}\otimes k[G_i]^* \ar^-F[ld] \\
		& 	k &
	} 	
	$$ 	
	where $\chi$ is the base change of $^*\psi^\ell_{i_1,\ldots,i_n}$ and the two vertical arrows are multiplication with $b$ (using the noncommutative multiplication of the tensor algebra). Then the top square commutes.
	As $\chi= {^*\psi^\ell_{i_1,\ldots,i_n,i}}$ and $F$ is multiplicatively coherent, the lower triangle commutes. So the whole diagram commutes, showing that $R_bF$ is multiplicatively coherent.	
	
	The argument for unital coherence and $H$-coherence is similar.
	 The proof for $L$ instead of $R$ is analogous.
\end{proof}

We can finally show that $k[G_1*_HG_2]$ is a subcoalgebra of $C(k[G_1]\oplus k[G_2])$. Let $f\in k[G_1*_HG_2]$. We have to show that $$\Delta(f)\in k[G_1*_HG_2]\otimes k[G_1*_HG_2]\subseteq C(k[G_1]\oplus k[G_2])\otimes C(k[G_1]\oplus k[G_2]).$$
From Lemma \ref{lemma: Hazewinkel} we know that $\Delta(f)$ can be written as $\Delta(f)=\sum_{i=1}^rg_i\otimes h_i$ with $g_i,h_i\in C(k[G_1]\oplus k[G_2])$, $\fun(g_1),\ldots,\fun(g_r)\in Rf$ and $\fun(h_1),\ldots,\fun(h_i)\in Lf$. Lemma \ref{lemma; coherence preserved under translation} implies that the functionals $\fun(g_1),\ldots,\fun(g_r),\fun(h_1),\ldots,\fun(h_i)$ are coherent. So Lemma \ref{lemma: coherence dualizes} implies that $g_1,\ldots,g_r,h_1,\ldots,h_r$ are coherent, i.e., they lie in $k[G_1*_HG_2]$ as desired.

So $k[G_1*_HG_2]$ is a subbialgebra of $C(k[G_1]\oplus k[G_2])$. The next step is to show that $k[G_1*_HG_2]$ is a Hopf algebra, i.e., we need to find the antipode. Define $$S\colon \hat{T}(k[G_1]\oplus k[G_2])\to \hat{T}(k[G_1]\oplus k[G_2])$$
by applying the antipodes $S_i\colon k[G_i]\to k[G_i]$ ($i=1,2$) on each of the factors and then reversing the order of the factors. So
$S\colon T_{i_1,\ldots,i_n}\to T_{i_n,\ldots,i_1}$ is given by 
$S(h_1\otimes\ldots\otimes h_n)=S_{i_n}(h_n)\otimes\ldots\otimes S_{i_1}(h_1)$. We would like to know that $S(C(k[G_1]\oplus k[G_2]))\subseteq C(k[G_1]\oplus k[G_2])$.

Recall that for a vector space $V$, we defined $\hat{T}V\hat{\otimes}\hat{T}V=\prod_{i,j\in\nn}V^{\otimes i} \otimes V^{\otimes j}$. Define $\tau\colon \hat{T}V\hat{\otimes}\hat{T}V\to \hat{T}V\hat{\otimes}\hat{T}V$ by $$\tau\colon V^{\otimes i}\otimes V^{\otimes j}\to V^{\otimes j}\otimes V^{i},\ a\otimes b\mapsto b\otimes a. $$ 

So in the matrix notation (\ref{eq: matrix notation}), the map $\tau$ transposes the matrix and then applies the flip to every matrix entry. From this description it is clear that $\tau$ maps $\hat{T}V\otimes\hat{T}V$ to $\hat{T}V\otimes\hat{T}V$.

For $f\in \hat{T}(k[G_1]\oplus k[G_2])$, a direct computation shows that $\Delta(S(f))=\tau((S\otimes S)(\Delta(f)))$. So if $f\in C(k[G_1]\oplus k[G_2])$, then $\Delta(S(f))\in \hat{T}(k[G_1]\oplus k[G_2])\otimes \hat{T}(k[G_1]\oplus k[G_2])$, i.e., $S(f)\in C(k[G_1]\oplus k[G_2])$.
Thus $S$ restricts to a map $S\colon C(k[G_1]\oplus k[G_2])\to C(k[G_1]\oplus k[G_2])$. The commutativity of the diagrams
$$
\xymatrix{
	T_{i_1,\ldots,i_n} \ar^-{\psi^\ell_{i_1,\ldots,i_n}}[rr] \ar_S[d] & & T_{i_1,\ldots,i_\ell,i_\ell,\ldots,i_n} \ar^S[d] \\
	T_{i_n,\ldots,i_1} \ar^-{\psi^{n-\ell+1}_{i_n,\ldots,i_1}}[rr] & & T_{i_n,\ldots,i_\ell,i_\ell,\ldots, i_1}
}
$$

$$
\xymatrix{
	T_{i_1,\ldots,i_n} \ar^-{\delta^\ell_{i_1,\ldots,i_n}}[rr] \ar_-S[d]
	& & T_{i_1,\ldots,i_{\ell-1},i_{\ell+1},\ldots,i_n} \ar^-S[d] \\
	T_{i_n,\ldots,i_1} \ar^-{\delta^{n-\ell+1}_{i_1,\ldots,i_n}}[rr]	& & T_{i_n,\ldots,i_{\ell+1},i_{\ell-1},\ldots,i_1}
}
$$
and 
$$
\xymatrix{
T_{i_1,\ldots,1,\ldots i_n} \ar_-S[d] \ar^-{\theta^\ell_{i_1,\ldots,1,\dots,i_n}}[rr] & & k[G_{i_1}]\otimes\ldots\otimes k[H]\otimes\ldots\otimes k[G_{i_n}] \ar[d] & & T_{i_1,\ldots, 2, \ldots, i_n}  \ar_-{\theta^\ell_{i_1,\ldots,2,\dots,i_n}}[ll] \ar^-S[d]\\
T_{i_n,\ldots,1,\ldots i_1} \ar^-{\theta^{n-\ell+1}_{i_n,\ldots,1,\dots,i_1}}[rr] & & k[G_{i_n}]\otimes\ldots\otimes k[H]\otimes\ldots\otimes k[G_{i_1}] & & T_{i_n,\ldots, 2, \ldots, i_1}	\ar_-{\theta^{n-\ell+1}_{i_n,\ldots,2,\dots,i_1}}[ll]
}
$$
show that indeed $S$ restricts to a map $S\colon k[G_1*_HG_2]\to k[G_1*_HG_2]$. To see that $S$ is an antipode for the bialgebra $k[G_1*_HG_2]$, we need to show that
\begin{equation}
\label{eq: for Hopf algebra}
\xymatrix{
	k[G_1*_HG_2] \ar^-\Delta[r] \ar_\epsilon[d] & k[G_1*_HG_2]\otimes k[G_1*_HG_2] \ar^-{S\cdot \id}[d] \\
	k \ar[r] & k[G_1*_HG_2]
}
\end{equation}
commutes.

Let $f\in 	k[G_1*_HG_2]$. Note that the product in $\hat{T}(k[G_1]\oplus k[G_2])$ is such that for $a\in T_{i_1,\ldots,i_n}$ and $b\in T_{j_1,\ldots,j_m}$ we have $ab=0$, unless $m=n$ and $i_1=j_1,\ldots,i_n=j_n$.
Thus, when trying to determine the image of $\Delta(f)$ under $S\cdot\id$, only the components of $\Delta(f)$ in $T_{i_n,\ldots,i_1}\otimes T_{i_1,\ldots,i_n}\subseteq \hat{T}(k[G_1]\oplus k[G_2])\hat{\otimes}\hat{T}(k[G_1]\oplus k[G_2])$ are relevant, as all other components map to zero under $S\cdot \id$. So the component $(S\cdot\id)(\Delta(f))_{i_1,\ldots,i_n}$ of $(S\cdot\id)(\Delta(f))$ in $T_{i_1,\ldots,i_n}$, is the image of the component $f_{i_n,\ldots,i_1,i_1,\ldots,i_n}$ of $f$ in $T_{i_n,\ldots,i_1,i_1,\ldots,i_n}$ under the map
$$
\xi_{i_1,\ldots,i_n}\colon T_{i_n,\ldots,i_1,i_1,\ldots,i_n}\to T_{i_1,\ldots,i_n},\quad a_n\otimes\ldots\otimes a_1\otimes b_1\otimes\ldots\otimes b_n\mapsto S_{i_1}(a_1)b_1\otimes\ldots\otimes S_{i_n}(a_n)b_n.
$$

We will show, by induction on $n$, that $\xi_{i_1,\ldots,i_n}(f_{i_n,\ldots,i_1,i_1,\ldots,i_n})=\epsilon(f)\in k$ as desired.
For $n=1$, $T_{i_1,i_1}=k[G_{i_1}]\otimes k[G_{i_1}]$, $T_{i_1}=k[G_{i_1}]$ and the commutative diagram
$$
\xymatrix{
	k[G_{i_1}]\otimes k[G_{i_1}] \ar^-{\xi_{i_1,i_i}=S_{i_1}\cdot\id}[rr] & & k[G_{i_1}] \\
	k[G_{i_1}]\ar^-{\psi^1_{i_1}=\Delta_{i_1}}[u] \ar^-{\epsilon_{i_1}}[rr] & & k \ar[u]
}
$$
shows that $\xi_{i_1,i_1}(f_{i_1,i_1})=\epsilon_{i_1}(f_{i_1})$. But for $n=1$, the condition $\delta^\ell_{i_1,\ldots,i_n}(f_{i_1,\ldots,i_n})=f_{i_1,\ldots,i_{\ell-1},i_{\ell+1},\ldots, i_n}$ means that $\epsilon_{i_1}(f_{i_1})=f_\emptyset=\epsilon(f)$, the degree zero component of $f$. This resolves the $n=1$ case. For $n>1$, the diagram
$$ 
\xymatrix{
	T_{i_n,\ldots,i_1,i_1,\ldots,i_n} \ar^-{\xi_{i_1,\ldots,i_n}}[rd] & \\
	T_{i_n,\ldots,i_1,\ldots,i_n} \ar^-{\psi^n_{i_n,\ldots,i_1,\ldots,i_n}}[u] \ar_{\delta^n_{i_n\ldots,i_1,\ldots,i_n}}[d] & T_{i_1,\ldots,i_n} \\
	T_{i_n,\ldots,i_2,i_2,\ldots,i_n} \ar^-{\xi_{i_2,\ldots,i_n}}[r] & T_{i_2,\ldots,i_n} \ar@{^(->}[u]
}
$$
commutes. (This is probably most conveniently seen by looking at what happens with the tuples of group elements when considering the dual morphisms.) So, by induction, 
\begin{align*}
\xi_{i_1,\ldots,i_n}(f_{i_n,\ldots,i_1,i_1,\ldots,i_n})&=\xi_{i_1,\ldots,i_n}(\psi^n_{i_n,\ldots,i_1,\ldots,i_n}(f_{i_n,\ldots,i_1,\ldots,i_n}))=\\
&=\xi_{i_2,\ldots,i_n}(\delta^n_{i_n,\ldots,i_1,\ldots,i_n}(f_{i_n,\ldots,i_1,\ldots,i_n}))=\xi_{i_2,\ldots,i_n}(f_{i_n,\ldots,i_2,i_2,\ldots,i_n})=\\
&=\epsilon(f).
\end{align*}
Thus (\ref{eq: for Hopf algebra}) commutes and $k[G_1*_HG_2]$ is indeed a Hopf algebra.

%
%
%
%

The map $\pi\colon C(k[G_1]\oplus k[G_2])\to k[G_1]\oplus k[G_2]$ can be restricted to $k[G_1*_HG_2]$ and composing with the projections on $k[G_1]\oplus k[G_2]$ yields maps $p_1\colon k[G_1*_HG_2]\to k[G_1]$ and $p_2\colon k[G_1*_HG_2]\to k[G_1]$.

It is clear that $p_i$ $(i=1,2)$ is a morphism of algebras. Compatibility with the counit follows from (\ref{eq: counit condition}) with $n=1$, whereas compatibility with the comultiplication follows from (\ref{eq: comultiplication condition}) with $n=1$.
The commutativity of the diagram
$$
\xymatrix{ 
& k[G_1*_H G_2] \ar_-{p_1}[ld] \ar^-{p_2}[rd] & \\
k[G_1] \ar_-{\alpha_1}[rd] &  & k[G_2] \ar^{\alpha_2}[ld] \\
& k[H] &	
}
$$
follows from (\ref{eq: H coherence}) with $n=1$.

\begin{prop} \label{prop: fibre product of Hopf algebras}
	The Hopf algebra $k[G_1*_HG_2]$ together with the projections $p_1,p_2$ is the fibre product of $\alpha_1\colon k[G_1]\to k[H]$ and $\alpha_2\colon k[G_2]\to k[H]$ in the category of Hopf algebras.
\end{prop}	
\begin{proof}
	Let $k[G]$ be a Hopf algebra with morphisms $q_1\colon k[G]\to k[G_1]$ and $q_2\colon k[G]\to k[G_2]$ of Hopf algebras such that $\alpha_1q_1=\alpha_2q_2$. We then have a morphism $\varphi=(q_1,q_2)\colon k[G]\to k[G_1]\oplus k[G_2]$ of $k$-algebras. By the universal property of $C(k[G_1]\oplus k[G_2])$, there exists a unique morphism $\tilde{\varphi}\colon k[G]\to C(k[G_1]\oplus k[G_2])$ of bialgebras such that $\pi\tilde{\varphi}=\varphi$.
	
	We claim that $\tilde{\varphi}$ maps $k[G]$ into $k[G_1*_HG_2]\subseteq C(k[G_1]\oplus k[G_2])$. Let $f\in k[G]$. Then
	$$\tilde{\varphi}(f)=(\epsilon_G(f),\varphi(f),\varphi^{\otimes 2}(\Delta_G^2(f)),\varphi^{\otimes 3}(\Delta_G^3(f)),\ldots).$$
	The commutative diagram
	$$
	\xymatrix{
	& k[G]^{\otimes n} \ar^-{\varphi^{\otimes n}}[r] & (k[G_1]\oplus k[G_2])^{\otimes n} \ar[r] & T_{i_1,\ldots,i_n} \ar^-{\psi^\ell_{i_1,\ldots,i_n}}[dd] \\ 
	k[G] \ar^-{\Delta_G^n}[ru] \ar_-{\Delta_G^{n+1}}[rd] & & & \\
		& k[G]^{\otimes (n+1)} \ar^-{\varphi^{\otimes(n+1)}}[r] & (k[G_1]\oplus k[G_2])^{\otimes (n+1)} \ar[r] & T_{i_1,\ldots,i_\ell,i_\ell,\ldots,i_n}
}
	$$
shows that $\tilde{\varphi}(f)$ is multiplicatively coherent. The commutative diagram
		$$
	\xymatrix{
		& k[G]^{\otimes n} \ar^-{\varphi^{\otimes n}}[r] & (k[G_1]\oplus k[G_2])^{\otimes n} \ar[r] & T_{i_1,\ldots,i_n} \ar^-{\delta^\ell_{i_1,\ldots,i_n}}[dd] \\ 
		k[G] \ar^-{\Delta_G^n}[ru] \ar_-{\Delta_G^{n-1}}[rd] & & & \\
		& k[G]^{\otimes (n-1)} \ar^-{\varphi^{\otimes(n-1)}}[r] & (k[G_1]\oplus k[G_2])^{\otimes (n-1)} \ar[r] & T_{i_1,\ldots,i_{\ell-1},i_{\ell+1},\ldots,i_n}
	}
	$$
	shows that $\tilde{\varphi}(f)$ is unitaly coherent.
Finally, the commutative diagram 
	$$
\xymatrix{
	& k[G]^{\otimes n} \ar^-{\varphi^{\otimes n}}[r] & (k[G_1]\oplus k[G_2])^{\otimes n} \ar[r] & T_{i_1,\ldots,1,\ldots,i_n} \ar^-{\theta^\ell_{i_1,\ldots,1,\ldots,i_n}}[d] \\ 
	k[G] \ar^-{\Delta_G^n}[ru] \ar_-{\Delta_G^{n}}[rd] & & &  k[G_{i_1}]\otimes\ldots\otimes k[H]\otimes\ldots\otimes k[G_{i_n}] \\
	& k[G]^{\otimes n} \ar^-{\varphi^{\otimes n}}[r] & (k[G_1]\oplus k[G_2])^{n} \ar[r] & T_{i_1,\ldots,2,\ldots,i_n} \ar_-{\theta^\ell_{i_1,\ldots,2,\ldots,i_n}}[u]
}
$$
shows that $\tilde{\varphi}(f)$ is $H$-coherent. So $\tilde{\varphi}(f)$ is coherent and we have a well-defined morphism $\tilde{\varphi}\colon k[G]\to k[G_1*_H G_2]$ with $p_1\tilde{\varphi}=q_1$ and $p_2\tilde{\varphi}=q_2$.
	
	If $\varphi'\colon k[G]\to k[G_1*_HG_2]$ is another morphism with $p_1\varphi'=q_1$ and $p_2\varphi'=q_2$, then $\pi\varphi'=\varphi$ and the universal property of $C(k[G_1]\oplus k[G_2])$ yields $\varphi'=\tilde{\varphi}$.
\end{proof}

\subsection{Properties of amalgamated free products}
\label{subsec: Properties of amalgamated free products}

In this section we discuss properties of amalgamated free products and their compatibility with other constructions like proalgebraic completion or base change. Throughout Section \ref{subsec: Properties of amalgamated free products} we assume that $H\to G_1$ and $H\to G_2$ are two morphisms of $k$-groups.

\begin{lemma} \label{lemma: quotient map}
	The canonical morphism $G_1*G_2\to G_1*_HG_2$ is a quotient map.
\end{lemma}
\begin{proof}
	The canonical morphism $G_1*G_2\to G_1*_HG_2$ is dual to the inclusion $k[G_1*_H G_2]\subseteq k[G_1*G_2]$ of subalgebras of $C(k[G_1]\oplus k[G_2])$.
\end{proof}

\begin{prop}
	If $G_1$ and $G_2$ are geometrically reduced, then $G_1*_H G_2$ is geometrically reduced.
\end{prop}
\begin{proof}
	According to Lemma \ref{lemma: quotient map} it suffices to show that $k[G_1*G_2]$ is geometrically reduced. Let $\overline{k}$ denote the algebraic closure of $k$. We have to show that $k[G_1*G_2]\otimes_k\overline{k}$ is reduced. Since $k[G_1*G_2]\otimes_k\overline{k}$ is a subring of $\hat{T}(k[G_1]\oplus k[G_2])\otimes_k\overline{k}$, it suffices to shows that $\hat{T}(k[G_1]\oplus k[G_2])\otimes_k\overline{k}$ is reduced. For $i=1,2$ we abbreviate $\overline{k}[G_i]=k[G_i]\otimes_k \overline{k}$. As $\hat{T}(k[G_1]\oplus k[G_2])\otimes_k\overline{k}$ injects into $\hat{T}(\overline{k}[G_1]\oplus\overline{k}[G_2])$, it suffices to see that $\hat{T}(\overline{k}[G_1]\oplus\overline{k}[G_2])$. But as a ring, $\hat{T}(\overline{k}[G_1]\oplus\overline{k}[G_2])$ is the product of the rings $(\overline{k}[G_1]\oplus\overline{k}[G_2])^{\otimes n}$. As $\overline{k}[G_1]\oplus\overline{k}[G_2]=\overline{k}[G_1]\times\overline{k}[G_2]$ is reduced, also $(\overline{k}[G_1]\oplus\overline{k}[G_2])^{\otimes n}$ is reduced (\cite[Chapter V, \S 15, No. 5, Theorem 3]{Bourbaki:Algebra2}) and therefore $\hat{T}(\overline{k}[G_1]\oplus\overline{k}[G_2])$ is reduced.
\end{proof}

\begin{lemma}
	The morphisms $G_1\to G_1*G_2$ and $G_2\to G_1*G_2$ are closed embeddings.
\end{lemma}
\begin{proof}
	It suffices to show that $G_1(R)\to (G_1*G_2)(R)$ is injective for every $k$-algebra $R$.
	Set $\f_1=\id\colon G_1\to G_1$ and let $\f_2\colon G_2\to G_1$ be the trivial morphism, i.e., the morphism that factors through the trivial group. By the universal property of $G_1*G_2$, there exists a morphism $\f\colon G_1*G_2\to G_1$ such that the composition  $G_1\to G_1*G_2\xrightarrow{\f} G_1$ is the identity. In particular, $G_1(R)\to (G_1*G_2)(R)$ is injective for every $k$-algebra $R$.	
%
\end{proof}

It may seem natural to expect that $G_i\to G_1*_HG_2$ ($i=1,2$) are closed embeddings if $H\to G_i$ are closed embeddings, since the analogous statement holds for abstract groups. However, the analogous statement does not even hold for profinite groups (\cite[Example 9.2.10]{RibesZalesskii:ProfiniteGroups}) and therefore one should not expect it for proalgebraic groups.

In a similar vein, still assuming that $H\to G_i$ are closed embeddings, one may hope that the natural map
$G_1(R)*_{H(R)}G_2(R)\to (G_1*_HG_2)(R)$ is injective for every $k$-algebra $R$, where $G_1(R)*_{H(R)}G_2(R)$ is the amalgamated free product of abstract groups. However, if $G_1(R)\to (G_1*_HG_2)(R)$ is not injective, then also  
$G_1(R)*_{H(R)}G_2(R)\to (G_1*_HG_2)(R)$ cannot be injective, because $G_1(R)\to G_1(R)*_{H(R)}G_2(R)$ is injective.


\bigskip

For a $k$-group $G$, let $G^{\et}$ be the maximal pro-\'etale quotient. It is canonically isomorphic to the $k$\=/group $\pi_0 (G)$ of connected components of $G$. Every morphism of $k$-groups $H \to G^{\et}$ factors uniquely $H \to H^{\et} \to G^{\et}$. Hence, for $k$-groups $H , G_1 , G_2$ and $G$, the following two commutative diagrams are equivalent
\[
\vcenter{\xymatrix{
H \ar[r] \ar[d] & G^{\et}_1 \ar[d] \\
G^{\et}_2 \ar[r] & G
}}
\quad \text{and} \quad 
\vcenter{\xymatrix{
H^{\et} \ar[r] \ar[d] & G^{\et}_1 \ar[d] \\
G^{\et}_2 \ar[r] & G \; .
}}
\]
Hence we have a canonical identification
\begin{equation}
\label{eq:cd8}
G^{\et}_1 \ast_H G^{\et}_2 = G^{\et}_1 \ast_{H^{\et}} G^{\et}_2 \; .
\end{equation}

\begin{prop}
\label{t2.8}
Given morphisms $\alpha_1 : H \to G_1$ and $\alpha_2 : H \to G_2$ of $k$-groups, the canonical morphism
\[
G_1 \ast_H G_2 \longrightarrow G^{\et}_1 \ast_H G^{\et}_2 = G^{\et}_1 \ast_{H^{\et}} G^{\et}_2
\]
induces an isomorphism
\begin{equation}
\label{eq:cd9}
(G_1 \ast_H G_2)^{\et} \silo (G^{\et}_1 \ast_{H^{\et}} G^{\et}_2)^{\et} \; .
\end{equation}
Equivalently we have an isomorphism
\begin{equation}
\label{eq:cd10}
\pi_0 (G_1 \ast_H G_2) \silo \pi_0 (\pi_0 (G_1) \ast_{\pi_0 (H)} \pi_0 (G_2)) \; .
\end{equation}
\end{prop}

\begin{proof}
Consider commutative diagrams of $k$-groups where $E$ is pro-\'etale
\[
\xymatrix{
H \ar[r]^{\alpha_1} \ar[d]_{\alpha_2} & G_1 \ar[d] \\
G_2 \ar[r] & E \, .
}
\]
We say that the diagram is a pro-\'etale push-out, if for every other such diagram with pro-\'etale $E'$ there is a unique morphism $E \to E'$ for which the obvious diagrams commute. A pro-\'etale push-out is unique up to a unique isomorphism. It is straightforward to check that both sides of \eqref{eq:cd9} give rise to pro-\'etale push-outs $E$. Hence \eqref{eq:cd9} is an isomorphism. 
\end{proof}

\begin{rem}
 It is not true that the amalgamated product of pro-\'etale $k$-groups is again pro-\'etale. Consider the representations
\[
\ZZ / 4 \to \GL_2 (\QQ) \; , \; 1 \mapsto \left( \begin{smallmatrix} 0 & -1 \\ 1 & 0 \end{smallmatrix} \right) \quad \text{and} \quad \ZZ / 6 \to \GL_2 (\QQ) \; , \; 1 \mapsto \left( \begin{smallmatrix} 1 & -1 \\ 1 & 0 \end{smallmatrix} \right) \; .
\]
Their images are known to generate the infinite subgroup $\Sl_2 (\ZZ)$ of $\GL_2 (\QQ)$. We have corresponding representations of $\QQ$-groups $(\ZZ / 4)_{\QQ} \to \GL_{2, \QQ}$ and $(\ZZ / 6)_{\QQ} \to \GL_{2 , \QQ}$ and hence an algebraic representation of the free product
\[
(\ZZ / 4)_{\QQ} \ast (\ZZ / 6)_{\QQ} \xrightarrow{\rho} \GL_{2, \QQ} \; .
\]
Its image on $\QQ$-valued points contains $\Sl_2 (\ZZ)$ because the latter is the image of the composition
\[
\ZZ / 4 \ast \ZZ / 6 \longrightarrow ((\ZZ / 4)_{\QQ} \ast (\ZZ / 6)_{\QQ}) (\QQ) \xrightarrow{\rho (\QQ)} \GL_2 (\QQ) \; .
\]
The algebraic quotients of pro-\'etale group schemes are finite. Hence $(\ZZ / 4)_{\QQ} \ast (\ZZ / 6)_{\QQ}$ cannot be pro-\'etale. 
\end{rem}

\begin{cor}
\label{t2.9}
If $G_1 , G_2$ are connected $k$-groups, then $G_1 \ast_H G_2$ is connected as well.
\end{cor}

\begin{proof}
By Proposition \ref{t2.8}, we have
\[
\pi_0 (G_1 \ast_H G_2) = \pi_0 (\pi_0 (G_1) \ast_{\pi_0 (H)} \pi_0 (G_2)) \; .
\]
An affine group scheme $G$ (over a field $k$) is connected if and only if $\pi_0 (G) = \spec k$ is the trivial group. Since the amalgamated product of two trivial $k$-groups is the trivial $k$-group by the universal property, the assertion follows.
\end{proof}


We next show that free products are compatible with the formation of \emph{free proalgebraic groups}. To this end, let us recall the definition of the free proalgebraic group $\Gamma_X$ on a set $X$ (\cite[Section 2.5]{Wibmer:FreeProalgebraicGroups}). Let $\overline{k}$ denote the algebraic closure of $k$ and let $G$ be a $k$-group. A map $X\to G(\overline{k})$ \emph{converges to $1$} if for every quotient map $G\to H$ with $H$ algebraic (i.e., of finite type over $k$) the composition $X\to G(\overline{k})\to H(\overline{k})$ maps all but finitely many elements of $X$ to the identity element of $H(\overline{k})$. The free proalgebraic group $\Gamma_X$ on $X$ comes equipped with a map $X\to\Gamma_X(\overline{k})$ that converges to $1$ and satisfies the following universal property: If $X\to G(\overline{k})$ is a map converging to $1$, then there exists a unique morphism $\Gamma_X\to G$ of $k$-groups such that
$$
\xymatrix{
X \ar[rr] \ar[rd] & & \Gamma_X(\overline{k}) \ar[ld] \\
& G(\overline{k}) &	
}
$$
commutes.

\begin{prop}
	Let $X$ and $Y$ be sets, then $\Gamma_X*\Gamma_Y\simeq \Gamma_{X\uplus Y}$.
\end{prop}
\begin{proof}
We will show that $\Gamma_{X\uplus Y}$ satisfies the universal property of $\Gamma_X*\Gamma_Y$. As the compositions $X\to X\uplus Y\to \Gamma_{X\uplus Y}(\overline{k})$ and  $Y\to X\uplus Y\to \Gamma_{X\uplus Y}(\overline{k})$ converge to $1$, we have induced morphisms $\Gamma_X\to \Gamma_{X\uplus Y}$ and $\Gamma_Y\to \Gamma_{X\uplus Y}$. Given morphisms $\Gamma_X\to G$ and $\Gamma_Y\to G$ of $k$-groups, we can define a map $X\uplus Y\to G(\overline{k})$ by sending an element of $X$ to its image under $X\to \Gamma_X(\overline{k})\to G(\overline{k})$ and an element of $Y$ to its image under $Y\to \Gamma_X(\overline{k})\to G(\overline{k})$. Then $X\uplus Y\to G(\overline{k})$ converges to $1$ and therefore induces a morphism $\Gamma_{X\uplus Y}\to G$. The diagrams
\begin{equation} \label{eq: 2 diagrams}
\xymatrix{ 
\Gamma_{X\uplus Y} \ar[rr] & & G & & \Gamma_{X\uplus Y} \ar[rr] & & G \\
& \Gamma_X \ar[ul] \ar[ur]	 & & & & \Gamma_Y \ar[ul] \ar[ur]	 & 
}
\end{equation} commute because their restrictions to $X$ respectively $Y$ commute. On the other hand, any two morphisms $\Gamma_{X\uplus Y}\to G$ making (\ref{eq: 2 diagrams}) commute, will have the same restriction to $X\uplus Y$ and therefore are equal.
\end{proof}

We next show that amalgamated free products are compatible with \emph{proalgebraic completion}. For a discrete group $\Gamma$, let $\Gamma^{\alg}$ be the proalgebraic completion (or proalgebraic hull) of $\Gamma$ over $k$. It is a $k$-group equipped with a morphism $\Gamma \to \Gamma^{\alg} (k)$ with the universal property that any linear representation $\Gamma \to \GL (V)$ on a finite dimensional $k$-vector space $V$ comes from an algebraic representation of $\Gamma^{\alg} \to \GL_V$ by composition
\[
\Gamma \longrightarrow \Gamma^{\alg} (k) \longrightarrow \GL_V (k) = \GL (V) \; .
\]
The $k$-group $\Gamma^{\alg}$ may be obtained as the Tannakian dual of the category $\Rep (\Gamma)$ of representations of $\Gamma$ on finite dimensional $k$-vector spaces, with its canonical fibre functor. For more background on proalgebraic completions see e.g., \cite[Section 2.1]{BassLubotzkyMagidMozes:TheProalgebraicCompletionOfRigidGroups}.

\begin{prop}
\label{t2.10}
Given morphisms of discrete groups $\Gamma \to \Gamma_1$ and $\Gamma \to \Gamma_2$, the following canonical morphism is an isomorphism:
\[
(\Gamma_1 \ast_{\Gamma} \Gamma_2)^{\alg} \silo \Gamma^{\alg}_1 \ast_{\Gamma^{\alg}} \Gamma^{\alg}_2 \; .
\]
\end{prop}

\begin{proof}
We have equivalences (Section \ref{sec:1})
\begin{align*}
\Rep (\Gamma^{\alg}_1 \ast_{\Gamma^{\alg}} \Gamma^{\alg}_2) & = \Rep (\Gamma^{\alg}_1) \times_{\Rep (\Gamma^{\alg})} \Rep (\Gamma^{\alg}_2) \\
& = \Rep (\Gamma_1) \times_{\Rep (\Gamma)} \Rep (\Gamma_2) \\
& = \Rep (\Gamma_1 \ast_{\Gamma} \Gamma_2) = \Rep ((\Gamma_1 \ast_{\Gamma} \Gamma_2)^{\alg}) \; .
\end{align*}
The assertion now follows from Tannakian duality. 
\end{proof}


Contrary to what one might expect, the amalgamated free product may not be compatible with base change. Indeed, if $H\to G_1$ and $H\to G_2$ are morphisms of $k$-groups and $K$ is a field extension of $k$, the commutative diagram
$$
\xymatrix{
	& (G_1*_H G_2)_K & \\
	G_{1,K} \ar[ru] & & G_{2,K} \ar[lu] \\
	& H_K \ar[lu], \ar[ru] &	
}
$$
yields a morphism $G_{1,K}*_{H_K}G_{2,K}\to (G_1*_H G_2)_K$. The dual morphism is the restriction of $\widehat{T}(k[G_1]\oplus k[G_2])\otimes_k K\to  \widehat{T}((k[G_1]\oplus k[G_2])\otimes_k K)$. The latter map is always injective but rarely surjective (\cite[Chapter II, \S 3.7, Corollary 3]{Bourbaki:Algebra1}). 
So $(G_1*_H G_2)_K$ is a quotient of $G_{1,K}*_{H_K}G_{2,K}$. To construct a concrete example, where the morphism $G_{1,K}*_{H_K}G_{2,K}\to (G_1*_H G_2)_K$ fails to be an isomorphism, we need some preparation. 

Assume, for now, that $k$ is an algebraically closed field of characteristic zero.
Let $C_2$ denote the (abstract) group with two elements. We first show that the set of isomorphism classes of irreducible finite dimensional $k$-linear representations of $C_2*C_2$ has cardinality $|k|$. Note that to specify a representation of $C_2$ on a $k$-vector space $V$ is equivalent to specifying a direct sum decomposition $V=V_1\oplus V_{-1}$, where the non-trivial element of $C_2$ acts trivially on $V_1$ and by multiplications with $-1$ on $V_{-1}$. To specify a representation of $C_1*C_2$ is thus equivalent to specifying two direct sum decompositions $V=V_1\oplus V_{-1}$ and $V=V'_1\oplus V'_{-1}$. 
The corresponding representation of $C_2*C_2$ is irreducible if and only if $V_i\cap V'_j=0$ for all $i,j\in \{1,-1\}$. For $a\in k\smallsetminus\{0,1\}$ let $V_a=k^2$ be the representation of $C_2*C_2$ determined by $V_1=k\begin{pmatrix}1 \\ 0 \end{pmatrix}$, $V_{-1}=k\begin{pmatrix}0 \\ 1 \end{pmatrix}$, $V_1'=k\begin{pmatrix}1 \\ 1 \end{pmatrix}$ and $V_{-1,a}'=k\begin{pmatrix}1 \\ a \end{pmatrix}$. The representations $V_a$ are irreducible and pairwise non-isomorphic. Indeed, if $a,b\in k\smallsetminus\{0,1\}$, and $f\colon V_a\to V_{b}$ is an isomorphism, then $f$ must stabilize $V_1$ and $V_{-1}$, i.e., $f$ is given by a diagonal matrix. As $f$ also stabilizes $V_1'$, $f$ is in fact given by a scalar matrix. As $f$ maps $V'_{-1,a}$ into $V'_{-1,b}$, we find that $a=b$.

Since the $V_a$'s are irreducible and pairwise non-isomorphic as representations of $C_1*C_2$, the $V_a$'s are also irreducible and pairwise non-isomorphic as representations of $(C_2*C_2)^{\alg}=(C_2*C_2)^{\alg,k}$. Every (finite dimensional) representation of a $k$-group $G$ embeds into $k[G]^n$ for some $n\geq 1$ (\cite{Waterhouse:IntroductiontoAffineGroupSchemes}). It follows that every irreducible representation of $G$ embeds into $k[G]$. In particular, every $V_a$ embeds into $k[(C_2*C_2)^{\alg}]$. As a sum of pairwise non-isomorphic irreducible representations is always a direct sum, it follows that the dimension $\dim_k k[(C_2*C_2)^{\alg}]$ of $k[(C_2*C_2)^{\alg}]$ as a $k$-vector space is at least $|k|$. 

On the other hand, the set $S$ of all morphisms $\rho\colon C_2*C_2\to\GL_{n(\rho)}(k)$ has at most cardinality $|k|$ and $(C_2*C_2)^{\alg}$ can be realized as a closed subgroup scheme of $\prod_{\rho \in S}\GL_{n(\rho),k}$ (see \cite[Section 2.1]{BassLubotzkyMagidMozes:TheProalgebraicCompletionOfRigidGroups}). So the dimension of $k[(C_2*C_2)^{\alg}]$ is at most $|k|$. Therefore $\dim_k k[(C_2*C_2)^{\alg}]=|k|$.

Let $\mu_2=\mu_{2,k}$ denote the $k$-group of second roots of unity, i.e., $\mu_2(R)=\{g\in R^\times|\ g^2=1 \}$ for every $k$-algebra $R$. Using Proposition \ref{t2.10} we find
$$\mu_{2}*\mu_{2}=C_2^{\alg}*C_2^{\alg}=(C_2*C_2)^{\alg}.$$ So $\dim_kk[\mu_2*\mu_2]=|k|$.
Therefore, if $k\subseteq K$ is an inclusion of algebraically closed fields of characteristic zero with $|k|<|K|$, then $\mu_{2,K}*\mu_{2,K}$ and $(\mu_{2,k}*\mu_{2,k})_K$ cannot be isomorphic because the coordinate ring of $\mu_{2,K}*\mu_{2,K}$ has $K$-dimension $|K|$, whereas the coordinate ring of $(\mu_{2,k}*\mu_{2,k})_K$ has $K$-dimension $|k|$.


%

\section{Local systems and a Seifer-van Kampen Theorem} \label{sec:3}

The main goal of this short section is to establish an analog of the Seifert-van Kampen theorem for the proalgebraic fundamental group.

Let $X$ be a topological space. We denote with $\loc(X)=\loc_k(X)$ the category of local systems of finite dimensional $k$-vector spaces on $X$. An object $\Eh$ of $\loc(X)$ is thus a sheaf of $k$-vector spaces on $X$, such that every point $x\in X$ has an open neighborhood $U$ with $\Eh |_U\simeq \underline{k^r}$, where $\underline{k^r}$ is the constant sheaf on $U$ associated to the vector space $k^r$ (and $r \ge 0$ may depend on $U$).

For an open subset $U$ of $X$ and $\Eh\in \loc(X)$, the restriction $\Eh|_U$ is in $\loc(U)$. Indeed $\Eh\rightsquigarrow \Eh |_U$ is a functor $\loc(X)\to \loc(U)$. 

\begin{lemma} \label{t3-1}
	Let $X$ be a topological space and let $U_1,U_2$ be open subsets of $X$ such that $X=U_1\cup U_2$.  Then the functor $\Eh\rightsquigarrow (\Eh |_{U_1},\Eh |_{U_2} , \id_{\Eh |_{U_1 \cap U_2}})$ from $\loc(X)$ to  $\loc(U_1)\times_{\loc(U_1\cap U_2)} \loc(U_2)$ is an equivalence of categories.
\end{lemma}
\begin{proof}
	As in Example \ref{t1-2} this is just the sheaf property. See \cite[\href{https://stacks.math.columbia.edu/tag/00AK}{Tag 00AK}]{stacks-project}.
\end{proof}

The espace \'etale of any $\Eh$ in $\uLoc (X)$ may be viewed as a vector bundle over $X$ with locally constant transition functions in $\GL_r (k)$ for some $r$ which may vary. Equivalently, if we give $k$ the discrete topology, these are the vector bundles with continuous $\GL_r (k)$-valued transition functions. The stalk $\Eh_x$ of $\Eh$ in a point $x \in X$ is isomorphic to the fibre in $x$ of the corresponding vector bundle. We have $\Eh_x = k^r$ for some $r = r(x) \ge 0$. If $X$ is connected, these ranks are independent of $x$ and $\uLoc_k (X)$ is a neutral Tannakian category. It is neutralized by the fibre functor $\omega_x : \uLoc (X) \to \nVec_k$, which sends local systems and their morphisms to the respective stalks, c.f. \cite[Proposition 2.1]{D}. In \cite{D}, the \emph{proalgebraic fundamental group} $\pi_k (X,x)$ of a connected pointed topological space $(X,x)$ was defined to be the Tannakian dual of $(\uLoc_k (X), \omega_x)$. Thus we have
\[
\uLoc_k (X) = \Rep (\pi_k (X,x)) \; .
\]
The previous constructions immediately give the following result:


\begin{theo}
\label{t32}
Let $U_1 , U_2$ be open subsets of a topological space $X$ such that $X=U_1\cup U_2$ and assume that $X , U_1 , U_2$ and $U_1 \cap U_2 \neq \emptyset$ are all connected. Then, for any point $x \in U_1 \cap U_2$ the commutative diagram of $k$-groups
\begin{equation}
\label{eq:31}
\vcenter{\xymatrix{
\pi_k (U_1 \cap U_2 , x) \ar[r] \ar[d] & \pi_k (U_1 , x) \ar[d] \\
\pi_k (U_2 , x) \ar[r] & \pi_k (X,x)
}}
\end{equation}
is a push-out diagram in the category of $k$-groups, i.e. we have an isomorphism
\[
\pi_k (X,x) \silo \pi_k (U_1 , x) \times_{\pi_k (U_1 \cap U_2 , x)} \pi_k (U_2 , x) \; .
\]
\end{theo}

\begin{proof}
By Lemma \ref{t3-1} the following diagram is a fibre product of categories
\[
\xymatrix{
\uLoc_k (X) \ar[r] \ar[d] & \uLoc_k (U_1) \ar[d] \\
\uLoc_k (U_2) \ar[r] & \uLoc_k (U_1 \cap U_2) \; .
}
\]
Note that all (restriction) functors commute with the respective fibre functors $\omega_x$ on the respective categories. Hence the diagram of representation categories induced by \eqref{eq:31} is a fibre product as well
\[
\xymatrix{
\uRep (\pi_k (X,x)) \ar[r] \ar[d] & \uRep (\pi_k (U_1 , x)) \ar[d] \\
\uRep (\pi_k (U_2 , x)) \ar[r] & \uRep (\pi_k (U_1 \cap U_2 , x)) \; .
}
\]
By Corollary \ref{t1-11} in the situation of Corollary \ref{t1-12} it follows that \eqref{eq:31} is a push-out diagram.
\end{proof}

\section{The proalgebraic fundamental group of a topological space is perfect} \label{sec:4}

Throughout Section \ref{sec:4} we assume that $k$ is a perfect field of characteristic $p>0$. An affine group scheme $G =\spec A$ over $k$ is \emph{perfect} if the $p$-power map $(\;)^p$ is bijective on $A$. In this section we give Tannakian criteria with applications for when the Tannaka dual group of a neutralized Tannakian category $(\Th , \omega)$ over $k$ is reduced resp. perfect. The main tool is the functor $\Fr_+$ from \cite{Coulembier:TannakianCategoriesInPositiveCharacteristic} and our criteria are implicitely contained in \cite{Coulembier:TannakianCategoriesInPositiveCharacteristic}. For convenience we give short self-contained proofs of the required properties of $\Fr_+$. In fact, there are two versions of $\Fr_+$ depending on the choice of a subgroup $S \subseteq \Sh_p$, the symmetric group on $p$ letters, and either can be used in the following. The two choices are $S = \Sh_p$ and $S = \langle \sigma \rangle = \ZZ / p$ where $\sigma$ is the cyclic permutation $\sigma = (1,2, \ldots,p)$. For an object $X$ of $\Th$ there is a natural action of $S$ on $X^{\otimes p}$ and $\Fr_+ (X)$ is the image in $\Th$ of the composion:
\[
H^0 (S , X^{\otimes p}) \hookrightarrow X^{\otimes p} \overset{[\,]}{\twoheadrightarrow} H_0 (S , X^{\otimes p}) \; .
\]
Here one sets
\[
H^0 (S , X^{\otimes p}) = \ker \Big((\sigma - \id)_{\sigma} : X^{\otimes p} \to \prod_{\sigma \in S} X^{\otimes p} \Big)
\]
and
\[
H_0 (S , X^{\otimes p}) = \coker \Big( \sum_{\sigma} (\sigma - \id)_{\sigma} : \bigoplus_{\sigma \in S} X^{\otimes p} \to X^{\otimes p} \Big) \; .
\]
For a morphism $\varphi$ in $\Th$ there is a canonical induced morphism $\Fr_+ (\varphi)$ and $\Fr_+$ becomes a functor.

Let $\Th^{(p)}$ be the category with the same objects as $\Th$ and morphisms
\[
\Hom_{\Th^{(p)}} (X , Y) = \Hom_{\Th} (X , Y) \otimes_{k , (\,)^p} k \; .
\]
These are $k$-vector spaces via multiplication on the second factor. We view $\Fr_+$ as a functor from $\Th^{(p)}$ to $\Th$ by setting $\Fr_+ (\varphi \otimes \lambda) = \lambda \Fr_+ (\varphi)$. Note that this is well defined since for $\mu \in k$ we have $\Fr_+ (\mu \varphi) = \mu^p \Fr_+ (\varphi)$. Thus $\Fr_+ : \Th^{(p)} \to \Th$ is compatible with scalar multiplication. It is not difficult to see that $\Fr_+$ is also additive and hence $k$-linear. E.g. for $S = \Sh_p$ and morphisms $\varphi , \psi : X \to Y$ we have $\Fr_+ (\varphi + \psi) = \Fr_+ (\varphi) + \Fr_+ (\psi)$ because as morphisms from $H_0 (\Sh_p , X^{\otimes p})$ to $H_0 (\Sh_p , Y^{\otimes p})$ we have:
\[
[(\varphi + \psi)^{\otimes p}] = \sum^p_{i=0} {p \choose i} [\varphi^{\otimes i} \otimes \psi^{\otimes p-i}] = [\varphi^{\otimes p}] + [\psi^{\otimes p}] \; .
\]
Define a faithful tensor functor $\omega^{(p)} : \Th^{(p)} \to \nVec_k$ by setting $\omega^{(p)} (X) = \omega (X) \otimes_{k , (\,)^p} k$ and $\omega^{(p)} (\varphi \otimes \lambda ) = \omega (\varphi) \otimes \lambda$. Then $(\Th^{(p)} , \omega^{(p)})$ becomes a neutralized Tannakian category over $k$. If $G = \spec A$ is the Tannakian dual of $(\Th , \omega)$, then $G^{(p)} = \spec A^{(p)}$ with $A^{(p)} = A \otimes_{k , (\,)^p} k$ may be identified with the Tannakian dual of $(\Th^{(p)} , \omega^{(p)})$ because a $k$-linear $G^{(p)}$-action on $\omega^{(p)}(X)$ 
can be identified with a $k$-linear $G$-action on $\omega (X)$.

 Let $F_{\rel} : G \to G^{(p)}$ be the relative Frobenius morphism. It is the $k$-linear morphism corresponding to the map of $k$-algebras $F^{\sharp}_{\rel} : A^{(p)} \to A$ sending $a \otimes \lambda$ to $\lambda a^p$. The following facts can be found in \cite{Coulembier:TannakianCategoriesInPositiveCharacteristic}: 

\begin{prop}
	 \label{t41}
\mbox{}
\begin{enumerate}
	\item  For $V$ in $\nVec_k$ the natural map
\[
t (V) : V \otimes_{k , (\,)^p} k \xrightarrow{\sim} \Fr_+ (V) \; , \; v \otimes \lambda \longmapsto \lambda [v \otimes \ldots \otimes v]
\]
is a $k$-linear isomorphism. \\
\item The maps in a) induce a natural isomorphism of functors $\omega^{(p)} = \omega \circ \Fr_+ : \Th^{(p)} \to \nVec_k$ and the functor $\Fr_+ : \Th^{(p)} \to \Th$ can be identified with the functor $\Rep (F_{\rel}) : \Rep (G^{(p)}) \to \Rep (G)$ between the corresponding representation categories.\\
\item $\Fr_+$ is an exact tensor functor. 
\end{enumerate}
\end{prop}

\begin{proof}
(i) For $S = \Sh_p$ the map $t (V)$ is additive by the binomial theorem. Additivity in the case $S = \langle \sigma \rangle = \ZZ / p$ is also not difficult to see. Compatibility with scalar multiplications is obvious so that $t (V)$ is $k$-linear. Since $t (k)$ is an isomorphism and the functors $\_ \otimes_{k, (\,)^p} k$ and $\Fr_+$ are additive on $\nVec_k$, it follows that $t (V)$ is an isomorphism for all $V$.\\
(ii) For $X$ in $\operatorname{Ob}\,\Th^{(p)} = \operatorname{Ob}\,\Th$ the isomorphisms $t (\omega (X))$ from a) give the first of the isomorphisms $\omega^{(p)} = \Fr_+ \circ \, \omega = \omega \circ \Fr_+$. The second one follows since $\omega$ is an exact $k$-linear $\otimes$-functor. The $G^{(p)}$-action $\rho^{(p)}$ on $\omega^{(p)} (X)$ corresponds to a $G$-action $\rho$ on $V = \omega (X)$. By functoriality, we get an induced $G$-action $\Fr_+ (\rho)$ on $\Fr_+ (V)$ and via the isomorphism $t (V)$ a $G$-action $\rho'$ on $\omega^{(p)} (X) = V \otimes_{k , (\,)^p} k$. We have to show that $\rho' = \rho^{(p)} \circ F_{\rel}$. Let $\Delta : V \to V\otimes A$
  be the $A$-comodule map corresponding to $\rho$ and $\Delta_{\Fr_+ (V)}$ the induced $A$-comodule map on $\Fr_+ (V)$. For $v \in V$ set $\Delta (v) = \sum_i v_i\otimes a_i$. Then we have on $v \otimes 1 \in V \otimes_{k , (\,)^p} k$
\begin{align*}
(\Delta_{\Fr_+ (V)} \circ t (V)) (v \otimes 1) & = \Delta_{\Fr_+ (V)} ([v \otimes \ldots \otimes v]) \\
& = \sum_{i_1 , \ldots , i_p}  [v_{i_1} \otimes \ldots \otimes v_{ip}]\otimes a_{i_1} \cdots a_{i_p}  \\
& = \sum_i  [v_i \otimes \ldots \otimes v_i]\otimes a^p_i \; .
\end{align*} 
The last equation holds because if the stabilizer $S_{\ui}$ of $\ui = (i_1 , \dots , i_p)$ in $S = \langle \sigma \rangle = \ZZ / p \subseteq \Sh_p$ is not $S$, i.e. if not all $i_1 , \ldots , i_p$ are equal, then $S_{\ui} = 1$ and we sum $|S| = p$-times the same element in $A \otimes \Fr_+ (V) \subseteq A \otimes H_0 (S , V^{\otimes p})$ which gives zero. For $S = \Sh_p$ this is all the more true since $H_0 (\Sh_p , V^{\otimes p})$ is a quotient of $H_0 (\ZZ / p , V^{\otimes p})$. On the other hand, we have
\begin{align*}
\lefteqn{(( t (V)\otimes \id_A) \circ (\id\otimes F^{\sharp}_{\rel}) \circ \Delta^{(p)}) (v \otimes 1)} \\
& = ((t(V)\otimes\id_A) \circ (\id\otimes F^{\sharp}_{\rel})) \Big( \sum_i  (v_i \otimes 1) \otimes (a_i \otimes 1) \Big) \\
& = (t(V)\otimes\id_A) \Big( \sum_i (v_i \otimes 1) \otimes a^p_i \Big) = \sum_i [v_i \otimes \ldots \otimes v_i]\otimes  a^p_i \; .
\end{align*}
Thus we have checked that $\rho' = \rho^{(p)} \circ F_{\rel}$ and b) is proved.\\
(iii) These assertions follow from the interpretation of $\Fr_+$ as $\Rep (F_{\rel})$ in (ii).
\end{proof}

\begin{cor}
\label{t42}
For a neutralized Tannakian category $(\Th , \omega)$ over a perfect field $k$ of characteristic $p > 0$ let $G$ be its Tannakian dual.
\begin{enumerate}
	\item  The group scheme $G$ is reduced if and only if the functor $\Fr_+ : \Th^{(p)} \to \Th$ is fully faithful and every subobject of $\Fr_+ (T)$ for some $T$ in $\operatorname{Ob}\, \Th^{(p)}$ is isomorphic to $\Fr_+ (T')$ for some subobject $T'$ of $T$.
\item The group scheme $G$ is perfect if and only if $\Fr_+$ is an equivalence of categories.
\end{enumerate}
\end{cor}

\begin{proof}
(i) 
Note that $G$ is reduced if and only if $(\;)^p$ is injective on $A$. This condition is equivalent to the map $F^{\sharp}_{\rel} : A^{(p)} \to A$ being injective, i.e. to the relative Frobenius $F_{\rel} : G \to G^{(p)}$ being faithfully flat, c.f. \cite[14.1 Theorem]{W}. By \cite[Proposition 2.21]{DeligneMilne:TannakianCategories} and Proposition \ref{t41}, (ii) this translates into the stated properties of $\Fr_+ = \Rep (F_{\rel})$. \\
(ii) $G$ is perfect if and only if $(\;)^p$ is an automorphism of $A$ as an $\Fa_p$-algebra. Since $A^{(p)} \cong A$ as $\Fa_p$-algebras this is equivalent to $F^{\sharp}_{\rel}$ being an isomorphism of $k$-algebras i.e. to $F_{\rel}$ being an isomorphism of $k$-group schemes. By Tannakian duality and Proposition \ref{t41} (ii) this is equivalent to $\Fr_+ = \Rep (F_{\rel})$ being an equivalence of categories.
\end{proof}

\begin{cor}
\label{t43}
	Let $k$ be a perfect field of positive characteristic $p$.
	\begin{enumerate}
		\item For a discrete group $\Gamma$ the proalgebraic completion $\Gamma^{\alg}$ over $k$ is a perfect group scheme over $k$.
		\item Let $X$ be a connected topological space and $x_0 \in X$ a point. Then the proalgebraic fundamental group $\pi_k (X , x_0)$ over $k$ is a perfect group scheme over $k$. 
	\end{enumerate}
\end{cor}

\begin{ex}
\label{t44}
For $\Gamma = \ZZ$ and $k$ algebraically closed of characteristic $p > 0$ one may find the calculation of $\ZZ^{\alg}$ in \cite[Annexe A]{S}. There is a canonical isomorphism:
\[
\ZZ^{\alg} = \spec k [k^{\times}] \times \ZZ_p \; .
\]
Here $\spec k [k^{\times}]$ is the diagonalizable group scheme on $k^{\times}$ over $k$ and $\ZZ_p$ denotes the unipotent, pro-\'etale group scheme obtained as the projective limit of the constant group schemes $\ZZ / p^n$. Since $k$ is perfect, it is clear that $\ZZ^{\alg}$ is perfect as well in accordance with the corollary.
\end{ex}

\begin{proof}
(i) By definition, $\Gamma^{\alg}$ is the Tannakian dual of the category $\Rep (\Gamma)$ of representations of $\Gamma$ on finite dimensional $k$-vector spaces $V$. The functor $\Fr_+$ sends the $\Gamma$-representation $V$ to the induced $\Gamma$-representation on the vector space $\Fr_+ (V)$. The $k$-linear automorphism $t (V)$ in Proposition~\ref{t41}~(i) is $\Gamma$-equivariant. It follows that the functor $\Fr_+ : \Rep (\Gamma)^{(p)} \to \Rep (\Gamma)$ is isomorphic to the functor $\_ \otimes_{k , (\,)^p} k$ which is an equivalence of categories. Now we use Corollary~\ref{t42}~(ii).\\
(ii) By definition $\pi_k (X , x_0)$ is the Tannakian dual of the category $\loc_k (X)$ of locally constant sheaves $\Eh$ of finite dimensional $k$-vector spaces on $X$ together with the fibre functor $\omega (\Eh) = \Eh_{x_0}$. The functor $\Fr_+$ sends $\Eh$ to the locally constant sheaf $\Fr_+ (\Eh)$. We define a map of sheaves
\[
t (\Eh) : \Eh \otimes_{k , (\,)^p} k \longrightarrow \Fr_+ (\Eh) 
\]
by sending $s \otimes \lambda$ for a local section $s$ of $\Eh$ to the image in $\Fr_+ (\Eh)$ of $\lambda [s \otimes \ldots \otimes s]$. For the stalk at a point $x \in X$ we have $t (\Eh)_x = t (\Eh_x)$. Since the $t (\Eh_x)$ are isomorphisms of $k$-vector spaces by Proposition \ref{t41} (i), it follows that $t (\Eh)$ is a $k$-linear isomorphism of sheaves. Hence the functor $\Fr_+ : \loc_k (X)^{(p)} \to \loc_k (X)$ is isomorphic to the functor $\_ \otimes_{k , (\,)^p} k$ which is an equivalence of categories. Now we use Corollary \ref{t42} (ii).
\end{proof}

\begin{remark}
In the remark after \cite[Theorem 2.5]{D} we had previously shown that $\pi_k (X , x_0)$ was reduced.
\end{remark}
\section{The proalgebraic fundamental group of the Kucharczyk-Scholze space} \label{sec:5}
Given a field $F$ of characteristic zero containing all roots of unity, and an injective homomorphism $\mu (F) \hookrightarrow \C^{\times}$, Kucharczyk and Scholze have defined a scheme $\eX_F$ over $\C$ and a canonical isomorphism \cite[Corollary 4.15]{KS}
\begin{equation}
\label{eq:5-8}
\pi^{\et}_1 (\eX_F (\C) , x) = G_F \; .
\end{equation}
Here $\eX_F (\C)$ carries the complex topology and $x \in \eX_F (\C)$ is a point. The \'etale fundamental group $\pi^{\et}_1 (Z , z)$ of a pointed connected topological space $(Z , z)$ is defined in \cite[Section 2.2]{KS} using the formalism of Grothendieck categories. It is a profinite group which classifies the finite coverings of $Z$. Finally $G_F$ is a version of the absolute Galois group of $F$ defined via a fibre functor related to $(\eX_F , x)$. The group $G_F$ is isomorphic to the usual Galois group $\Aut_F (\oF)$ where $\oF$ is an algebraic closure of $F$, the isomorphism being unique up to an inner isomorphism. By \cite[Theorem~3.6]{D}, the \'etale fundamental group $\pi^{\et}_1 (Z,z)$ viewed as a group-scheme over a field $k$ is canonically isomorphic to the maximal pro-\'etale quotient of the proalgebraic group $\pi_k (Z,z)$ i.e.
\[
\pi_k (Z,z)^{\et} \silo \pi^{\et}_1 (Z,z) \; .
\]
In particular, we have
\[
\pi_k (\eX_F (\C) , x)^{\et} \silo \pi^{\et}_1 (\eX_F (\C) , x) = G_F \; .
\]
Thus the group of connected components of $\pi_k (\eX (\C) , x)$ may be identified with $G_F$. The connected component $\pi_k (\eX_F (\C) , x)^0$ is abelian and we determine its structure. We also discuss relations with the motivic Galois groups of $F$ and $\oF$. For the calculation of $\pi_k (\eX_F (\C) , x)$ we use the same approach as in \cite{KS} applied to $\pi_k$ instead of $\pi^{\et}_1$.

Given a connected topological space $Z$ and a continuous path $\gamma : [0,1] \to Z$ with $\gamma (0) = z_0$ and $\gamma (1) = z_1$ we obtain a $\otimes$-isomorphism $\omega_{\gamma} : \omega_{z_0} \to \omega_{z_1}$ between the fibre functors $\omega_{z_i} : \uLoc_k (Z) \to \nVec_k$ for $i = 0,1$ as follows: For $\Eh$ in $\uLoc_k (Z)$ the locally constant sheaf $\gamma^{-1} \Eh$ is constant on $[0,1]$. Evaluation therefore provides isomorphisms
\[
\omega_{z_0} (\Eh) = \Eh_{z_0} \underset{\ev_{z_0}}{\xleftarrow{\sim}} \Gamma ([0,1] , \gamma^{-1} \Eh) \underset{\ev_{z_1}}{\xrightarrow{\sim}} \Eh_{z_1} \cong \omega_{z_1} (\Eh) \; .
\]
The natural transformation $\omega_{\gamma} = \ev_{z_1} \circ \ev^{-1}_{z_0}$ depends only on the homotopy class of $\gamma$. Let
\[
\Aut^{\otimes} (\omega_{z_0}) \underset{L_{\omega_{\gamma}}}{\xrightarrow{\sim}} \Iso^{\otimes} (\omega_{z_0} , \omega_{z_1}) \underset{R_{\omega_{\gamma}}}{\xleftarrow{\sim}} \Aut^{\otimes} (\omega_{z_1})
\]
be defined by left resp. right translation of the point $\omega_{\gamma} \in \Iso^{\otimes} (\omega_{z_0} , \omega_{z_1}) (k)$. We obtain an isomorphism of $k$-groups
\[
\gamma_k = R^{-1}_{\omega_{\gamma}} \circ L_{\omega_{\gamma}} : \pi_k (Z,z_0) \silo \pi_k (Z, z_1) \; .
\]

\begin{prop} \label{t51}
Let $f,g : X \to Z$ be continuous maps between connected topological spaces. Let $H : X \times [0,1] \to Z$ be a homotopy between $f$ and $g$. For a fixed point $\tx \in X$ consider the path $\gamma : [0,1] \to Z$ from $f (\tx)$ to $g (\tx)$ defined by $\gamma (t) = H (\tx , t)$. Then the following diagram of $k$-groups commutes
\[
\xymatrix{
 & \pi_k (Z , f (\tx)) \ar[dd]^{\wr \; \gamma_k}\\
\pi_k (X , \tx) \ar[ur]^{f_{\ast}} \ar[dr]_{g_{\ast}} & \\
 & \pi_k (Z , g (\tx)) \; . 
}
\]
\end{prop}

For the proof we need a lemma on local systems. The formulation is a little bit involved because apart from connectedness the topological spaces $X , Z$ in Proposition \ref{t51} are arbitrary.

\begin{lemma} \label{t52}
Consider continuous maps $Y\overset{\displaystyle\xrightarrow{\pi}}{\underset{S}{\longleftarrow}} X$ between connected topological spaces such that $\pi \circ s = \id$. Assume that for all $x \in X$ the following conditions hold:
\begin{enumerate}
	\item  The fibre $\pi^{-1} (x)$ is a quasi-compact subspace of $Y$ and any two different points of $\pi^{-1} (x)$ have disjoint neighborhoods in $Y$.
	\item For any open subset $V$ of $Y$ with $\pi^{-1} (x) \subset V$ there is an open neighborhood $x \in U \subset X$ with $\pi^{-1} (U) \subset V$.
	\item All local systems in $\uLoc_k (\pi^{-1} (x))$ are constant.
\end{enumerate}
	Then for any local system $\Fh$ in $\uLoc_k (Y)$ the natural map of sheaves $\pi_{\ast} \Fh \to s^{-1} \Fh$ is an isomorphism.	
\end{lemma}

\begin{ex} \label{t53}
\rm Let $K$ be a compact, path-connected, locally path-connected and simply connected topological space. Let $X$ be a connected topological space. Then conditions (i)-(iii) in Lemma \ref{t52} are satisfied for $Y = X \times K \overset{\displaystyle \xrightarrow{\pi}}{\underset{s}{\longleftarrow}} X$, where $s$ is any continuous section of the projection $\pi$.
\end{ex}

Namely, (i) is clear and (iii) follows from the equivalences
\[
\uLoc_k (\pi^{-1} (x)) \cong \uLoc_k (K) \cong \Rep (\pi_1 (K,t)) \; , \;t \in K
\]
which are valid by our hypothesis, and the fact that $\pi_1 (K,t)$ is trivial. As for (ii), let $V \subset Y$ be open with $V \supset \pi^{-1} (x) = \{ x \} \times K$. For any point $t \in K$, there are open neighborhoods $U_t$ of $x$ in $X$ and $W_t$ of $t$ in $K$ such that $V \supset U_t \times W_t$. Since $K$ is compact, there are finitely many points $t_1 , \ldots , t_n \in K$ such that $W_{t_1} , \ldots , W_{t_n}$ cover $K$. Then $U = U_{t_1} \cap \ldots \cap U_{t_n}$ is an open neighborhood of $x$ in $X$ with $V \supset U \times W_{t_i}$ for all $i$, and hence $V \supset U \times K = \pi^{-1} (U)$. 
\bigskip

\noindent \textit{Proof of Lemma \ref{t52}.} We recall the definition of the map $\pi_{\ast} \Fh \to s^{-1} \Fh$. Let $s^0 \Fh$ be the presheaf on $X$ with sections
\[
(s^0 \Fh) (U) = \colim_{V \supset s (U)} \Fh (V) \; .
\]
The relation $\pi \circ s = \id$ implies that $\pi^{-1} (U) \supset s (U)$. The resulting natural maps
\[
(\pi_* \Fh) (U) = \Fh (\pi^{-1} (U)) \longrightarrow (s^0 \Fh) (U)
\]
give a map of presheaves $\pi_* \Fh \to s^0 \Fh$ and hence a map of sheaves $\pi_* \Fh \to s^{-1} \Fh$. It suffices to check that the induced maps on the stalks are isomorphisms
\[
(\pi_* \Fh)_{x} = \colim_{U \ni x} \Fh (\pi^{-1} (U)) \longrightarrow (s^{-1} \Fh)_x = \Fh_{s (x)} \; .
\]
For this, note that because of (ii) we have
\[
(\pi_* \Fh)_x = \colim_{V \supset \pi^{-1} (x)} \Fh (V) \; .
\]
Condition (i) allows us to apply \cite[\href{https://stacks.math.columbia.edu/tag/09V3}{Tag 09V3}]{stacks-project}
to deduce that
\[
\colim_{V \supset \pi^{-1} (x)} \Fh (V) = \Gamma (\pi^{-1} (x) , \Fh \, |_{\pi^{-1} (x)}) \; .
\]
Finally, since $\Fh \, |_{\pi^{-1} (x)}$ is a constant local system by condition (iii), the evaluation map in the point $s (x) \in \pi^{-1} (x)$ gives an isomorphism
\[
\Gamma (\pi^{-1} (x) , \Fh \, |_{\pi^{-1} (x)}) \underset{\ev_{s (x)}}{\silo} \Fh_{s (x)} = (s^{-1} \Fh)_x \; .
\]
The composition of these three isomorphisms is the map $(\pi_* \Fh)_x \to (s^{-1} \Fh)_x$ above and we are done. \beweisende

Now we can turn to the \textit{proof} of Proposition \ref{t51}. Using the homotopy $H$ we first construct an isomorphism of $\otimes$-functors
\[
\uH : f^{-1} \cong g^{-1} : \uLoc_k (Z) \longrightarrow \uLoc_k (X) \; .
\]
For $\Eh$ in $\uLoc_k (Z)$ the isomorphism
\[
\uH (\Eh) : f^{-1} (\Eh) \silo g^{-1} (\Eh)
\]
is obtained as follows. Consider the diagram
\[
\xymatrix{
X \times [0,1] \ar[r]^-{\pi} & X \ar@<1ex>[l]^-{s_1} \ar@<-2ex>[l]_-{s_0} 
}
\]
where $\pi$ is the projection and $s_0 (x) = (x,0) , s_1 (x) = (x,1)$. Setting $Y = X \times [0,1]$ and $s = s_0$ or $s_1$, by Example \ref{t53} we may apply Lemma \ref{t52} to the sheaf $\Fh = H^{-1} (\Eh)$. We obtain canonical isomorphisms
\[
f^{-1} (\Eh) = s^{-1}_0 \Fh \xleftarrow{\sim} \pi_* \Fh \silo s^{-1}_1 \Fh = g^{-1} (\Eh) \; .
\]
Using these we define the isomorphism $\uH (\Eh)$. Its fibre in a point $x$ is a functorial isomorphism
\[
\uH (\Eh)_x : \Eh_{f (x)} = (f^{-1} \Eh)_x \silo (g^{-1} \Eh)_x = \Eh_{g (x)} \; ,
\]
so that $\uH (\;)_x \in \Iso^{\otimes} (\omega_{f (x)} , \omega_{g (x)})$. By construction, we have $\uH (\;)_{\tx} = \omega_{\gamma}$ as defined above, where $\gamma : [0,1] \to Z$ is the path $\gamma (t) = H (\tx , t)$. For the commutativity of the diagram we have to show that $L_{\omega_{\gamma}} \circ f_* = R_{\omega_{\gamma}} \circ g_*$ i.e. that for any $k$-algebra $R$ and any
\[
\alpha \in \pi_k (X , \tx) (R) = \uAut^{\otimes} (\omega_{\tx , R})
\]
we have
\[
\omega_{\gamma} \circ f_* (\alpha) = g_* (\alpha) \circ \omega_{\gamma} : \omega_{f (\tx)} \longrightarrow \omega_{g (\tx)} \; .
\]
In other words, we claim that for all $\Eh$ in $\uLoc_k (Z)$:
\[
(\uH (\Eh)_{\tx} \otimes \id_R) \circ \alpha (f^{-1} (\Eh)) = \alpha (g^{-1} (\Eh)) \circ (\uH (\Eh)_{\tx} \otimes \id_R) \; .
\]
This holds since $\uH (\Eh) : f^{-1} (\Eh) \to g^{-1} (\Eh)$ is a morphism in $\uLoc_k (X)$ and since $\alpha$ is a natural transformation. \beweisende

As in \cite[section 2.2]{KS} we obtain the following consequences:

\begin{cor} \label{t54}
Let $H$ be a pointed homotopy of the pointed connected topological spaces $(X , \tx)$ and $(Z , \tz)$ and set $f = H (\; , 0)$ and $g = H (\; , 1)$. Then
\[
f_* = g_* : \pi_k (X , \tx) \longrightarrow \pi_k (Z , \tz) \; .
\]
\end{cor}

\begin{proof}
For the path $\gamma$ in Proposition \ref{t51} we have $\gamma (t) = H (\tx , t) = \tz$ since $H$ is a pointed homotopy. Hence the corresponding isomorphism $\omega_{\gamma} : \Eh_{\tz} \to \Eh_{\tz}$ is the identity for all $\Eh$ in $\uLoc_k (Z)$ and therefore the map $\gamma_k$ in Proposition \ref{t51} is the identity as well.
\end{proof}

\begin{cor} \label{t55}
Let $(X, x) \overset{\xrightarrow{\;f\;}}{\underset{g}{\longleftarrow}} (Y,y)$ be pointed maps of pointed topological spaces such that $f \circ g$ and $g \circ f$ are homotopic to the identities via pointed homotopies. Then $X$ is connected if and only if $Y$ is connected. If this is the case, the functors $\uLoc_k (X) \overset{\xrightarrow{f^{-1}}}{\underset{g^{-1}}{\longleftarrow}} \uLoc_k (Y)$ are quasi-inverse equivalences of categories and $\pi_k (X,x) \overset{\xrightarrow{f_*}}{\underset{g_*}{\longleftarrow}} \pi_k (Y,y)$ are isomorphisms of $k$-groups which are inverse to each other.
\end{cor}

\begin{proof}
The assertion on connectedness is proved in the corresponding corollary \cite[Corollary~2.18]{KS}. The rest is a formal consequence of Corollary \ref{t54}.
\end{proof}

Consider a field $F$ of characteristic zero containing all roots of unity in an algebraic closure $\oF$ of $F$ and fix an injective homomorphism $\iota : \mu (F) = \mu (\oF) \hookrightarrow \C^{\times}$. Let $W_{\rat} (F)$ denote the ring of rational Witt vectors of $F$, c.f. \cite[section 4.1]{KS}. The Teichm\"uller map $[\;] : F^{\times} \to W_{\rat} (F)$ is multiplicative and we obtain ring homomorphisms
\[
\ZZ [\mu (F)] \xrightarrow{[\;]} W_{\rat} (F) \quad \text{and} \quad \ZZ [\mu (F)] \xrightarrow{\iota} \C \; ,
\]
defined on the group ring of $\mu (F)$. Set
\[
\eX_F = \spec (W_{\rat} (F) \otimes_{\ZZ [\mu (F)]} \C ) \; .
\]
The complex points of this scheme over $\spec \C$ carry the ``complex topology'' induced by the topology of $\C$, c.f. \cite[section 5.2]{KS}. By definition, we have
\[
\eX_F (\C) = \{ \alpha \in \Hom (W_{\rat} (F) , \C) \mid \alpha \circ [\;] = \iota \} \; .
\]
The topological space $\eX_F (\C)$ is connected and it has the following Galois-theoretic description. Let $G_F = \Aut_F (\oF)$ be the absolute Galois group of $F$. The Teichm\"uller map $\oF^{\times} \to W_{\rat} (\oF)$ induces an isomorphism $\ZZ \oF^{\times} \silo W_{\rat} (\oF)$ and hence an isomorphism
\[
\ZZ [\oF^{\times}]^G \silo W_{\rat} (\oF)^G = W_{\rat} (F) \; .
\]
This implies that
\[
\eX_{\oF} (\C) = \{ \chi : \oF^{\times} \longrightarrow \C^{\times} \mid \chi \, |_{\mu (\oF)} = \iota \} \; .
\]
By \cite[Lemma 4.9]{KS}, the natural map
\[
\Hom (\ZZ [\oF^{\times}] , \C) / G \silo \Hom (\ZZ [\oF^{\times}]^G , \C)
\]
is a bijection. This gives a bijection
\[
\eX_F (\C) = \{ \chi : \oF^{\times} \to \C^{\times} \mid \chi \, |_{\mu (\oF)} = \iota \} / G \; .
\]
Note that since $\mu (\oF) = \mu (F)$, the Galois group $G$ fixes $\iota$. This bijection is actually a homeomorphism if the set of characters is given the topology of pointwise convergence and its quotient by $G$ the quotient topology. Kucharczyk and Scholze prove that the categories of finite \'etale coverings of $\eX_F$ and $\eX_F (\C)$ are naturally equivalent and that their isomorphic \'etale fundamental groups with respect to a $\C$-valued point $x$ can be canonically identified with the absolute Galois group of $F$ if the latter is defined using a certain fibre functor defined by $x$. We will not consider the scheme $\eX_F$ in the following but only the connected topological space $\eX_F (\C)$ with the complex topology. Let $(\oF^{\times})^{\vee} = \Hom (\oF^{\times} , S^1)$ be the Pontrjagin dual of the discrete group $\oF^{\times}$ and consider the closed subspace
\[
X_{\oF} = \{ \chi \in (\oF^{\times})^{\vee} \mid \chi \, |_{\mu (F)} = \iota \} \; .
\]
The group $G_F$ acts by homeomorphisms and we endow the quotient $X_F = G \setminus X_{\oF}$ with the quotient topology. By \cite[Proposition 5.1]{KS}, the $G_F$ action on $X_{\oF}$ is proper and free and the space $X_F$ is a non-empty, connected, compact Hausdorff-space. By \cite[Theorem~5.7,~ii)]{KS} the natural $G_F$-equivariant projection $\eX_{\oF} (\C) \to X_{\oF} , \chi \mapsto \chi / |\chi|$ and the induced projection $\eX_F (\C) \to X_F$ are (strong) deformation retractions. Here we view $X_{\oF}$ as a closed subspace of $\eX_{\oF} (\C)$ via $X_{\oF} = X_{\oF} \times \{ 0 \}$ and similarly for $X_F$. From Corollary \ref{t55} we thus obtain:

\begin{cor} \label{t56}
For any point $x \in X_F \subset \eX_F (\C)$ the above map $\eX_F (\C) \to X_F$ induces an isomorphism of $k$-groups
\[
\pi_k (\eX_F (\C) , x) \silo \pi_k (X_F , x) \; .
\]
\end{cor}

It is sufficient therefore to determine $\pi_k (X_F , x)$. For a finite extension $F' / F$ in $\oF$ the natural map $X_{F'} \to X_F$ is a finite \'etale covering of compact connected Hausdorff spaces, c.f. \cite[Theorem 5.2]{KS}. Let $\ox \in X_{\oF} \subset \eX_{\oF} (\C)$ be a point and let $x_{F'}$ be its image in $X_{F'}$. By \cite[Theorem 3.6]{D} we have a natural isomorphism
\[ 
\pi_k (X_F , x)^0 = \varprojlim_{(Y,y)} \pi_k (Y, y)
\]
where $(Y,y)$ runs over the directed set of isomorphism classes of pointed finite \'etale coverings of $(X_F , x)$. According to \cite[Theorem 5.2]{KS}, the above coverings $(X_{F'} , x_{F'}) \to (X_F , x)$ form a cofinal system, so that we obain an isomorphism
\begin{equation}
\label{eq:5-9}
\pi_k (X_F , x)^0 = \varprojlim_{F' / F} \pi_k (X_{F'} , x_{F'}) \; .
\end{equation}
Since we are dealing with compact Hausdorff spaces, the natural map
\[
X_{\oF} \longrightarrow \varprojlim_{F' / F} X_{\oF} / \gal (\oF / F') = \varprojlim_{F' / F} X_{F'}
\]
is a homeomorphism. Applying \cite[Theorem 4.3]{D} we get an isomorphism
\begin{equation}
\label{eq:5-10}
\pi_k (X_{\oF} , \ox) = \varprojlim_{F' / F} \pi_k (X_{F'} , x_{F'}) \; .
\end{equation}

Combining \eqref{eq:5-9} and \eqref{eq:5-10} with \cite[Theorem 3.6]{D} and the \cite[Theorem 5.2]{KS} identification of $\pi_1 (X_F , x)$ with $G_F$, we obtain the following result:

\begin{cor} \label{t57}
For any point $\ox \in X_{\oF}$ mapping to $x \in X_F$ we have a canonical short exact sequence of affine $k$-groups
\[
1 \longrightarrow \pi_k (X_{\oF} , \ox) \longrightarrow \pi_k (X_F , x) \longrightarrow (G_F)_{/ k} \longrightarrow 1 \; .
\]
There are canonical isomorphisms
\[
\pi_k (X_{\oF} , \ox) = \pi_k (X_F , x)^0 \quad \text{and} \quad \pi_k (X_F , x)^{\et} = \pi^{\et}_1 (X_F , x)_{/ k} = (G_F)_{/ k} \; .
\]
\end{cor}

Here, for a profinite group $G = \varprojlim_i G_i$ with $G_i$ finite, $G_{/ k} = \varprojlim_i (G_{i/ k})$ denotes the pro-\'etale $k$-group obtained as the limit of the finite constant $k$-groups $G_{i / k}$ attached to the $G_i$'s. The global sections $\Gamma (G_{/ k} , \oo)$ of $G$ consists of the continuous maps $G \to k$ where $k$ carries the discrete topology. Note that for the $k$-rational points, we have
\[
\pi_k (X_F , x)^{\et} (k) = \pi^{\et}_1 (X_F , x) = G_F \; .
\]
Next we show that the connected $k$-group $\pi_k (X_{\oF} , \ox)$ is abelian. The inclusion $\mu_F \hookrightarrow \oF^{\times}$ induces a short exact sequence of compact abelian topological group
\[
1 \longrightarrow (\oF^{\times} / \mu_F)^{\vee} \longrightarrow (\oF^{\times})^{\vee} \xrightarrow{\pi} \mu^{\vee}_F \longrightarrow 1 \; .
\]
Here $\pi$ is the map which restricts a character $\chi : \oF^{\times} \to S^1$ to $\mu_F$. All fibres of $\pi$ are homeomorphic and in particular $X_{\oF}$, the fibre of our fixed embedding $\iota : \mu_F \hookrightarrow S^1$ is homeomorphic to $V^{\vee}$, the Pontrjagin dual of the $\QQ$-vectorspace $V = \oF^{\times} / \mu_F$. Write $V = \varinjlim_{\Gamma} \Gamma$ as a filtered union of finite free $\ZZ$-submodules. Then $V^{\vee} = \varprojlim_T T$ where $T = \Gamma^{\vee} \cong (S^1)^{\rk \Gamma}$ is a torus. Let $y \in V^{\vee}$ be the image of $\ox$ under a homeomorphism $X_{\oF} \silo V^{\vee}$ and let $y_T$ be the image of $y$ in $T$. Using \cite[Theorems 4.1 and 4.3]{D} we obtain isomorphisms of $k$-groups
\[
\pi_k (X_{\oF} , \ox) \cong \pi_k (V^{\vee}, y) = \varprojlim_T \pi_k (T , y_T) = \varprojlim_T \pi_1 (T , y_T)^{\alg} \; .
\]
Since $\pi_1 (T , y_T)$ is commutative, it follows that $\pi_k (X_{\oF}, \ox)$ is commutative as well. If $k$ is perfect, then it follows from \cite[Theorem 16.13]{Milne} that there is a canonical isomorphism of $k$ groups
\begin{equation}
\label{eq:5-11}
\pi_k (X_{\oF} , \ox) \silo \pi^u_k (X_{\oF} , \ox) \times \pi^{\red}_k (X_{\oF} , \ox) \; .
\end{equation}
Here $\pi^u_k (X_{\oF}, \ox)$ is the pro-unipotent radical of $\pi_k (X_{\oF}, \ox)$ and $\pi^{\red}_k (X_{\oF} , \ox)$, the maximal pro-reductive quotient of $\pi_k (X_{\oF} , \ox)$ can be identified with the greatest $k$-subgroup of $\pi_k (X_{\oF} , \ox)$ of multiplicative type. Both $\pi^u_k (X_{\oF} , \ox)$ and $\pi^{\red}_k (X_{\oF} , \ox)$ are commutative and connected. The cohomologies of $X_{\oF}$ with coefficients in $\ZZ / n$ and $\QQ$ have been calculated in \cite{KS}. We recall the argument for $H^1$ and coefficients in any abelian group $A$. 

Recall the pro-torus $V^{\vee} = \varprojlim_T T$ above where $V = \oF^{\times} / \mu_F$. We have
\[
H^1 (V^{\vee} , \uA) = \varinjlim_T H^1 (T , \uA) \; .
\]
It follows that translation by an element of $V^{\vee}$ induces the identity on $H^1 (V^{\vee} , \uA)$ since this is well known for the finite dimensional tori $T$. Hence the homeomorphisms $V^{\vee} \silo X_{\oF} , \alpha \mapsto \alpha \chi$ for $\chi \in X_{\oF}$ all induce the same isomorphism, which we will view as an identification
\begin{equation}
\label{eq:5-12}
H^1 (X_{\oF} , \uA) = H^1 (V^{\vee} , \uA) \; .
\end{equation}
We have 
\begin{align*}
H^1 (T , \uA) & = H^1 (\Gamma^{\vee} , \uA) = \Hom (\pi_1 (\Gamma^{\vee}) , A) \\
& = \Hom (\Hom (\Gamma , \ZZ) , A) = \Gamma \otimes A 
\end{align*}
and hence
\begin{equation}
 \label{eq:5-13}
 H^1 (X_{\oF}, \uA) = \varinjlim_T H^1 (T , \uA) = \oF^{\times} / \mu_F \otimes A \; .
\end{equation}
The category of commutative unipotent groups $U$ over $k \supset \QQ$ is anti-equivalent to the category of $k$-vector spaces by sending $U$ to $Y (U) = \Hom_k (U , \GG_a)$. Given $Y (U)$ we recover $U$ as
\[
U = \Hom_k (Y (U) , \GG_{a/k}) \; ,
\]
where by definition for all $k$-schemes $S$ we have
\[
\Hom_k (Y (U) , \GG_{a/k}) (S) := \Hom_k (Y (U) , \Gamma (S , \oo)) \; .
\]
Here, on the right $\Hom_k$ means $k$-linear maps. Again using \cite[Proposition 4.4]{D}, we obtain
\begin{equation}
\label{eq:5-14}
Y ( \pi^u_k (X_{\oF} , \ox)) = \Hom_k (\pi_k (X_{\oF} , \ox) , \GG_{a/k}) = H^1 (X_{\oF} , \uk) \; ,
\end{equation}
and then
\begin{equation}
\label{eq:5-15}
\pi^u_k (X_{\oF} , \ox) = \Hom_k (H^1 (X_{\oF} , \uk) , \GG_{a/k}) \; .
\end{equation}
We now turn to the multiplicative part. Since we want to apply \cite[Proposition 4.4]{D} which gives a formula for $\Hom_k (\pi_k (X_{\oF} , \ox) , \GG_m)$, but not for $\Hom_{k^{\sep}} (\pi_k (X_{\oF} , \ox) \otimes_k k^{\sep} , \GG_m)$, we restrict our discussion to the case where $k = \ok$. Then the category of multiplicative groups over $k$ is equivalent to the category of abelian groups by sending $H$ to $X^* (H) = \Hom_k (H , \GG_m)$. The group $H$ is recovered from $X^* (H)$ as 
\[
H = \Hom (X^* (H) , \GG_{m/k}) \; .
\]
Here the right hand side is defined by its $S$-valued points, which are the groups
\[
\Hom (X^* (H) , \GG_{m/k}) (S) := \Hom (X^* (H) , \Gamma (S , \oo^{\times})) \; .
\]
Using \cite[Proposition 4.4]{D}, we find
\[
X^* (\pi^{\red}_k (X_{\oF} , \ox)) = \Hom_k (\pi_k (X_{\oF} , \ox) , \GG_{m/k}) = H^1 (X_{\oF} , \uk^{\times}) \; 
\]
and therefore
\begin{equation}
\label{eq:5-16}
\pi^{\red}_k (X_{\oF} , \ox) = \Hom (H^1 (X_{\oF} , \uk^{\times}) , \GG_{m/k}) \; .
\end{equation}
Combining \eqref{eq:5-14}, \eqref{eq:5-15} with \eqref{eq:5-16} we get the following

\begin{cor}
\label{t58}
In the situation of Corollary \ref{t57} we have canonical isomorphisms of commutative $k$-groups:
\[
\pi_k (X_{\oF} , \ox) \silo \pi^u_k (X_{\oF} , \ox) \times \pi^{\red}_k (X_{\oF} , \ox) \; , \quad \text{if $k$ is perfect}
\]
with
\[
\pi^u_k (X_{\oF} , \ox) = \Hom_k (\oF^{\times} / \mu_F \otimes_{\QQ} k , \GG_{a / k})\; , \quad \text{if} \; k \supset \QQ 
\]
and
\[
\pi^{\red}_k (X_{\oF} , \ox) = \Hom (\oF^{\times} / \mu_F \otimes_{\QQ} k^{\times} / \mu_k , \GG_{m/k}) \; , \quad \text{if} \; k = \ok \supset \QQ \; .
\]
\end{cor}

\begin{rem}
	 Since the character group $\oF^{\times} / \mu_F \otimes_{\ZZ} k^{\times} = \oF^{\times} / \mu_F \otimes_{\QQ} k^{\times} / \mu_F$ of $\pi^{\red}_k (X_{\oF} , \ox)$ is a $\QQ$-vector space, it has no torsion and hence $\pi^{\red}_k (X_{\oF} , \ox)$ is connected. The unipotent part $\pi^u_k (X_{\oF} , \ox)$ being connected, it follows that $\pi_k (X_{\oF} , \ox)$ is connected. We know this of course, since $\pi_k (X_{\oF} , \ox) = \pi_k (X_F , x)^0$ by Corollary \ref{t57}.
\end{rem}	
\begin{rem}
 Let $G \to G''$ be a surjective i.e. faithfully flat map of affine $k$-groups. Equivalently, the corresponding map of Hopf-algebras $A'' \to A$ is injective. For algebraically closed fields $k$ the induced map $G (k) \to G'' (k)$ is surjective, c.f. \cite[III, \S 3, Collaire 7.6]{DemazureGabriel:GroupesAlgebriques}. 
  In particular, the map $G (k) \to \pi_0 (G) (k)$ is surjective.
\end{rem}

It follows from Corollary \ref{t57} that for $k = \ok$ we have a short exact sequence of pro-discrete groups
\begin{equation}
 \label{eq:5-17}
 1 \longrightarrow \pi^u_k (X_{\oF} , \ox) (k) \times \pi^{\red}_k (X_{\oF} , \ox) (k) \longrightarrow \pi_k (X_F , x) (k) \longrightarrow G_F \longrightarrow 1 \; .
\end{equation}
If additionally $\car k = 0$, then using Corollary \ref{t58}, we get
\[
\pi^u_k (X_{\oF} , \ox) (k) = \Hom_k (\oF^{\times} / \mu_F \otimes_{\QQ} k , k) = \Hom_{\QQ} (\oF^{\times} / \mu_F , k) 
\]
and
\[
\pi^{\red}_k (X_{\oF} , \ox) (k) = \Hom (\oF^{\times} / \mu_F \otimes_{\QQ} k^{\times} / \mu_F , k^{\times}) \; .
\]
Going through the constructions, one sees that the $G_F$-action on these groups via \eqref{eq:5-17} is the same as the one induced by the natural $G_F$-action on $\oF^{\times} / \mu_F$. The actions are continuous for the pro-discrete topologies on $\pi^u_k (X_{\oF} , \ox) (k)$ and $\pi^{\red}_k (X_{\oF} , \ox) (k)$ but not for the discrete topologies (the $G_F$-orbits do not have to be finite). Thus, although these groups are $\QQ$-vector spaces, it is not clear whether the extension \eqref{eq:5-17} is trivial i.e. a semi-direct product or not. The corresponding question is open for the extension of affine groups in Corollary \ref{t57} as well.

\medskip

The \'etale fundamental group of $X_F$ is isomorphic to the Galois group of $F$. We now discuss the question, whether the proalgebraic fundamental group $\pi_k (X_F , x)$ is related to the motivic Galois group $G_{\Mh_F (k)}$ if $k \supset \QQ$. We assume that our embedding $\iota : \mu_F \hookrightarrow \C$ is the restriction of an embedding $\iota : \oF \hookrightarrow \C$, which then also provides a point $\ox$ of $X_{\oF}$. Let $\Mh_F (k) , \Mh_{\oF} (k)$ be the Tannakian categories of mixed motives over $F$ with coefficients in $k$ neutralized by the Betti-realization via $\iota$. Thus e.g. for $M = H^w (X)$, with $X$ a variety over $\oF$ the fibre functor sends $M$ to $H^w_B (X \otimes_{\oF , \iota} \C , \uk)$. We can think of motives in terms of realizations as in \cite{Jannsen}. By definition, $G_{\Mh_F (k)}$ resp. $G_{\Mh_{\oF} (k)}$ are the Tannakian duals of $\Mh_F (k)$ resp. $\Mh_{\oF} (k)$ with respect to this fibre functor. There is a canonical short exact sequence of $k$-groups
\[
1 \longrightarrow G^0_{\Mh_F (k)} \longrightarrow G_{\Mh_F (k)} \longrightarrow G_{F / k} \longrightarrow 1 \; ,
\]
where $G_{\Mh_{\oF} (k)} = G^0_{\Mh_F (k)}$ canonically. As explained in the introduction of \cite{KS} there should be a ``better'' space $\tX_F$ mapping to $X_F$ and we would expect $\pi_k (\tX_F , \tx)$ to be closely related to $G_{\Mh_F (k)}$. This suggests a map $G_{\Mh_F (k)} \longrightarrow \pi_k (X_F, x)$ of $k$-groups and hence a map
\[
G_{\Mh_{\oF} (k)} = G^0_{\Mh_F (k)} \longrightarrow \pi_k (X_F , x)^0 = \pi_k (X_{\oF} , \ox) \; .
\]
It would map unipotent radicals to unipotent radicals and induce map between the maximal pro-reductive quotients. Since $\pi_k (X_{\oF} , \ox)$ is commutative, we would therefore get maps
\begin{equation}
\label{eq:5-18}
(G^u_{\Mh_{\oF} (k)})^{\ab} \longrightarrow \pi^u_k (X_{\oF} , \ox) 
\end{equation}
and
\begin{equation}
\label{eq:5-19}
(G^{\red}_{\Mh_{\oF} (k)})^{\ab} \longrightarrow \pi^{\red}_k (X_{\oF} , \ox) \; .
\end{equation}
We will now construct a canonical non-trivial map \eqref{eq:5-18} for any field $k \supset \QQ$. To do so is equivalent to giving a map of $k$-vector spaces
\begin{equation}
\label{eq:5-20}
Y (\pi^u_k (X_{\oF} , \ox)) \longrightarrow Y ((G^u_{\Mh_{\oF} (k)} )^{\ab}) = \Hom (G^u_{\Mh_{\oF} (k)} , \GG_a) \; .
\end{equation}
First note that by \eqref{eq:5-14} and \eqref{eq:5-13} we have
\begin{equation}
\label{eq:5-21} Y (\pi^u_k (X_{\oF} , \ox)) = \oF^{\times} / \mu_F \otimes k \; .
\end{equation}
Next, any $a \in \oF^{\times}$ defines a motive $E_a = H^1 (\GG_m , \{ 1,a \}) (1)$ over $\oF$. It can also be described as $E_a = H^1 (X_a) (1)$, where $X_a$ is the curve obtained from $\GG_{m, \oF}$ by identifying $1$ and $a$. In Deligne's theory of $1$-motives, we have $E_a = [ \ZZ \xrightarrow{\phi} \GG_m]$, where $\phi (1) = a$, c.f. \cite[\S\,10]{Deligne}. The motive $E_a$ fits into a canonical extension
\[
0 \longrightarrow \QQ (0) \longrightarrow E_a \longrightarrow \QQ (1) \longrightarrow 0 \quad \text{in} \; \Mh_{\oF} (\QQ) \; .
\]
Extending coefficients to $k$ we get an injective homomorphism
\begin{equation}
\label{eq:5-22}
\oF^{\times} / \mu \otimes k \hookrightarrow \Ext^1_{\Mh_{\oF} (k)} (k (0) , k (1)) \; , \; (a \; \operatorname{mod} \; \mu) \otimes 1 \mapsto E_a \otimes k \; .
\end{equation}
The motives $k (0)$, resp. $k (1)$ correspond to the trivial representation of $G_{\Mh_{\oF}} (k)$ on $\GG_a$ over $k$ resp. to the representation denoted $\GG_a (1)$ on $\GG_a$ by the character $\chi : G_{\Mh_{\oF}} (k) \to \GG_{m /k}$, dual to the full embedding of $\langle k (1) \rangle^{\otimes}$ into $\Mh_{\oF} (k)$. Then we have
\begin{equation}
\label{eq:5-23}
\Ext^1_{\Mh_{\oF} (k)} (k (0) , k (1)) = \Ext^1_{G_{\Mh_{\oF}} (k)} (\GG_a , \GG_a (1)) = H^1 (G_{\Mh_{\oF} (k)} , \GG_a (1)) \; .
\end{equation}
In the middle, extensions are taken within the category of finite dimensional representations of $G_{\Mh_{\oF} (k)}$ on $k$-vector spaces. For the Hochschild cohomology on the right compare \cite[Ch. 15]{Milne}. Consider the canonical extension of $k$-groups
\[
1 \longrightarrow G^u_{\Mh_{\oF} (k)} \longrightarrow G_{\Mh_{\oF} (k)} \longrightarrow Q \longrightarrow 1 \; .
\]
Here $Q = G^{\red}_{\Mh_{\oF} (k)}$ can be interpreted as the Tannakian dual of the category of (direct sums of) pure motives over $\oF$ with coefficients in $k$. We now follow the argument in \cite[Appendix~C4]{Jannsen}. Since Hochschild cohomology of reductive groups over fields of characteristic zero vanishes in positive degrees, the Hochschild-Serre spectral sequence
\[
H^i (Q , H^j (G^u_{\Mh_{\oF} (k)} , \GG_a (1)) \Rightarrow H^{i+j} (G_{\Mh_{\oF} (k)} , \GG_a (1))
\]
degenerates into isomorphisms
\[
H^j (G^u_{\Mh_{\oF} (k)} , \GG_a (1))^Q = H^j (G_{\Mh_{\oF} (k)} , \GG_a (1)) \; .
\]
In particular, using \eqref{eq:5-23} we may identify:
\begin{equation}
\label{eq:5-24}
\Ext^1_{\Mh_{\oF} (k)} (k (0) , k (1)) = H^1 (G^u_{\Mh_{\oF} (k)} , \GG_a (1))^Q \; .
\end{equation}
The character $\chi$ above factors over $Q$, the maximal pro-reductive quotient of $G_{\Mh_{\oF} (k)}$, and hence the restriction of $\chi$ to $G^u_{\Mh_{\oF} (k)}$ is trivial. Hence we have
\begin{equation}
 \label{eq:5-25}
 H^1 (G^u_{\Mh_{\oF} (k)} , \GG_a (1))^Q = H^1 (G^u_{\Mh_{\oF} (k)} , \GG_a) (1)^Q = \Hom (G^u_{\Mh_{\oF} (k)} , \GG_a) (1)^Q \; .
\end{equation}
The group on the right is the $\chi^{-1}$ eigenspace of the $Q$-action on $\Hom (G^u_{\Mh_{\oF} (k)} , \GG_a)$. In any case, we get the following inclusion by combining \eqref{eq:5-24} and \eqref{eq:5-25}
\begin{equation}
\label{eq:5-26}
\Ext^1_{\Mh_{\oF} (k)} (k (0) , k (1)) \subset \Hom (G^u_{\Mh_{\oF} (k)} , \GG_a) \; .
\end{equation}
The promised map \eqref{eq:5-20} is obtained as the following composition
\[
Y (\pi^u_k (X_{\oF} , \ox)) \overset{\eqref{eq:5-21}}{=} \oF^{\times} / \mu_F \otimes k \overset{\eqref{eq:5-22}}{\hookrightarrow} \Ext^1_{\Mh_{\oF} (k)} (k (0) , k (1)) \overset{\eqref{eq:5-26}}{\hookrightarrow} \Hom (G^u_{\Mh_{\oF} (k)} , \GG_a) \; . 
\]
(The map \eqref{eq:5-22} is conjectured to be an isomorphism.)

Contrary to \eqref{eq:5-18}, there does not seem to be a non-trivial map \eqref{eq:5-19}. The group $G^{\red}_{\Mh_{\oF}} (k)$ is the motivic Galois group for pure motives over $\oF$ with coefficients in $k$, and
\begin{equation}
\label{eq:5-27}
\Hom_k ((G^{\red}_{\Mh_{\oF} (k)})^{\ab} , \GG_m) = \Hom_k (G^{\red}_{\Mh_{\oF} (k)} , \GG_m) 
\end{equation}
is the group of isomorphism classes or rank $1$ motives over $\oF$ with coefficients in $k$. It follows that for any $k = \ok \supset \QQ$, \eqref{eq:5-19} would induce a homomorphism from the $\QQ$-vector space
\[
\Hom (\pi^{\red}_k (X_{\oF} , \ox) , \GG_m) = \Hom (\pi_k (X_{\oF} , \ox) , \GG_m) = H^1 (X_{\oF} , \uk^{\times}) = \oF^{\times} / \mu_F \otimes k^{\times}
\]
to the group \eqref{eq:5-27}, i.e. a map
\begin{equation}
\label{eq:5-28}
\oF^{\times} / \mu_F \otimes \QQ^{\times} \longrightarrow \Hom_k ((G^{\red}_{\Mh_{\oF} (k)})^{\ab} , \GG_m) \; .
\end{equation}
For $k = \oQ$, the $\oQ$-group
\[
S = (G^{\red}_{\Mh_{\oF} (\oQ)})^{\ab}
\]
is the connected Serre group of $F$ over $\oQ$. Its character group $\Hom_{\oQ} (S , \GG_m)$ is known \cite[\S\,7]{Serre} and does not depend on $F$! It does not contain any infinitely divisible elements except zero and hence any map \eqref{eq:5-28} has to be trivial. Note in this regard that as pointed out in \cite{KS} the Steinberg relations do not hold in the rational cohomology of $X_F$ (hence the above mentioned suggestion of a better space $\tX_F$ mapping to $X_F$). Morover no information about weights is built into the category $\uLoc_k (X_F)$ as we saw in our construction of the map \eqref{eq:5-18}. Both deficits, and also that $\pi_k (X_F , x)^0$ is commutative show that $X_F$ is quite far from a topological space whose proalgebraic fundamental group is more deeply related to the motivic Galois group of $F$.

\bibliographystyle{alpha}
 \bibliography{bibdata}
\end{document}